\documentclass[a4paper, 11pt]{amsart} 
\usepackage[
  margin=1.9cm,
  includefoot,
  footskip=30pt,
]{geometry}
\usepackage{amsmath,amssymb,amscd}
\usepackage{amsfonts}
\usepackage[english]{babel}
\usepackage[leqno]{amsmath}
\usepackage{amssymb,amsthm}
\usepackage{mathrsfs} 
\usepackage{amscd}
\usepackage{enumerate}
\usepackage{epsfig}
\usepackage{relsize}
\usepackage{layout}
\usepackage[usenames,dvipsnames]{xcolor}
\usepackage[backref=page]{hyperref}
\usepackage{tikz}
\hypersetup{
 colorlinks,
 citecolor=Green,
 linkcolor=Red,
 urlcolor=Blue}
\usepackage[matrix,arrow,tips,curve]{xy}
\input{xy}
\xyoption{all}
\definecolor{darkgreen}{rgb}{0.0, 0.7, 0.0}
\definecolor{purple}{rgb}{0.5, 0.0, 0.5}
\definecolor{red}{rgb}{0.8, 0.2, 0.0}

\newtheorem{thm}{Theorem}[section]
\newtheorem{bthm}{Theorem}
\newtheorem{bcor}{Corollary}
\newtheorem{lemma}[thm]{Lemma}

\newtheorem{prop}[thm]{Proposition}

\newtheorem{claim}[thm]{Claim}

\numberwithin{equation}{section}
\theoremstyle{definition}
\newtheorem{defi}[thm]{Definition}

\newtheorem{notation}[thm]{Notation}

\theoremstyle{remark}
\newtheorem{remark}[thm]{Remark}

\newtheorem{example}[thm]{Example}
\newcommand{\Z}{\mathbb{Z}}
\newcommand{\C}{\mathbb{C}}

\newcommand{\R}{\mathbb{R}}

\newcommand{\Pic}{\operatorname{Pic}}

\def \Im{{\rm Im}}

\def \P{\mathbb{P}}

\def \F{\mathcal F}
\def \N {\mathcal N}

\def\I{{\mathcal J}}
\def \L{\mathcal L}
\def \E{\mathcal E}
\def \G{\mathcal G}
\def \H{\mathcal H}

\def\O{\mathcal O}

\def\M0{\mathcal M^0}

\def\mapright#1{\smash{\mathop{\longrightarrow}\limits^{#1}}}

\DeclareMathOperator{\Sing}{{Sing}}

\DeclareMathOperator{\Proj}{{Proj}}
\DeclareMathOperator{\Sym}{{Sym}}

\def\B{\mathbf{B}}

\newcommand{\rk}{\operatorname{rank}}

\begin{document}

\title[Ulrich bundles with $c_2(\E)^2=0$ and connectedness of Ulrich subvarieties]{Ulrich bundles with $c_2(\E)^2=0$ and connectedness of Ulrich subvarieties}

\author[V. Buttinelli, A.F. Lopez, R. Vacca]{Valerio Buttinelli*, Angelo Felice Lopez and Roberto Vacca**}

\address{\hskip -.43cm Valerio Buttinelli, Dipartimento di Matematica ``Guido Castelnuovo", Sapienza Universit\`a di Roma, Piazzale Aldo Moro 5, 00185 Roma, Italy. email: {\tt valerio.buttinelli@uniroma1.it}}

\address{\hskip -.43cm Angelo Felice Lopez, Dipartimento di Matematica e Fisica, Universit\`a di Roma
Tre, Largo San Leonardo Murialdo 1, 00146, Roma, Italy. e-mail {\tt angelo.lopez@uniroma3.it}}

\address{\hskip -.43cm Roberto Vacca, Dipartimento di Matematica, Universit\`a di Roma Tor Vergata, Via della Ricerca Scientifica, 00133 Roma, Italy. email: {\tt vacca@mat.uniroma2.it}}

\thanks{*The author thanks the ``Progetti di Avvio alla Ricerca 2024" of the University of Rome La Sapienza}

\thanks{The second and third authors were partially supported by the GNSAGA group of INdAM} 

\thanks{**Work supported by the MIUR Excellence Department Project MatMod@TOV awarded to the Department of Mathematics of the University of Rome Tor Vergata.}


\begin{abstract} 
We give an almost complete classification of Ulrich bundles $\E$ with $c_2(\E)^2=0$ on a variety $X$ of dimension $n \ge 4$. Moreover, we show that there are strong constraints on the geometry of $X$ and we study disconnected Ulrich subvarieties.
\end{abstract}

\maketitle

\section{Introduction}


Let $X \subset \P^N$ be a smooth irreducible variety of dimension $n$ and let $\E$ be an Ulrich bundle on $X$, that is such that $H^i(\E(-p))=0$ for $1 \le p \le n$. It is nowadays well-known that, in the presence of an Ulrich bundle, several geometric consequences can be deduced about $X$, see for example \cite{be, es, cmrpl, ls}. A classical way to study vector bundles is through their degeneracy loci, see for example \cite[\S 14]{fu}. In the spirit of the Hartshorne-Serre correspondence, in  \cite[Thm.~3.1]{cfk}, \cite[Thm.~1]{lr1}, \cite[\S 2.2]{v} was considered a family of subvarieties, called Ulrich subvarieties, having several special properties and whose existence is equivalent to have an Ulrich bundle. As Ulrich subvarieties are codimension two degeneracy loci of an Ulrich bundle $\E$, hence their class in the Chow group of $X$ represents $c_2(\E)$, it is not a surprise that their geometrical properties are related to vanishing of suitable Chern polynomials. For example, as seen in \cite[Cor.~1]{blv}, their connectedness is controlled by the non-vanishing of $c_3(\E)$.

The leitmotif of our work is to show that non-connectedness of Ulrich subvarieties poses very severe restrictions on $X$, to the point of reaching classification results. As it will turn out, an important family of Ulrich bundles $\E$ to be studied is the one with $c_2(\E) \ne 0$ and $c_2(\E)^2=0$. In the latter case, Ulrich subvarieties move on a smooth divisor on $X$ (as shown in \cite[Lemma 4.5]{blv}) and very often cover $X$, see Lemma \ref{mov}. Moreover, we will be able to give several geometrical information about them. 

In order to study Ulrich bundles $\E$ on $X$ with $c_2(\E)^2=0$, the condition being obvious for $n \le 3$, the first case to be considered is when $n \ge 4$ and $c_1(\E)^4=0$. The latter, in fact, turns out to be stronger than $c_2(\E)^2=0$, as shown in the following classification. We will use the notion of linear Ulrich triple, see Definition \ref{not4}.

\begin{bthm} \hskip 3cm
\label{c1^4=0}

Let $X \subset \P^N$ be a smooth irreducible variety of dimension $n \ge 4$ and let $\E$ be a rank $r$ bundle on $X$. If $n \ge 5$, then $\E$ is Ulrich with $c_1(\E)^4=0$ if and only if $(X,\O_X(1),\E)$ is one of the following:
\begin{itemize}
\item [(i)] $(\P^n,\O_{\P^n}(1),\O_{\P^n}^{\oplus r})$.
\item [(ii)] A linear Ulrich triple over a variety of dimension $b$ with $1 \le b \le 3$ and $c_2(\E)^2=0$.
\end{itemize}
Now assume that $n=4$ and $\E$ is Ulrich with $c_1(\E)^4=0$. Then $c_2(\E)^2=0$ and:
\begin{itemize}
\item[(iii)] If $c_2(\E)=0$ and $r \ge 2$, then $(X,\O_X(1),\E)$ is either $(\P^4,\O_{\P^4}(1),\O_{\P^4}^{\oplus r})$ or a linear Ulrich triple over a curve.
\item[(iv)] If $c_2(\E)\ne0$, then $(X,\O_X(1),\E)$ is one of the cases (ii1) and (iii)-(x) of \cite[Thm.~1]{lms1}.
\end{itemize}
\end{bthm}

For examples of linear Ulrich triples as in the above theorem, see \cite{lo, lms1} and Example \ref{secondo}. An example of (iv) in Theorem \ref{c1^4=0}, such that $(X,\O_X(1))$ is not a linear $\P^k$-bundle, is given in \ref{esbs}.

Theorem \ref{c1^4=0} is also of independent interest, as it contributes to the study of non-big Ulrich bundles \cite{lm, lms1, lms2}, because when $n \ge 4$ and $c_1(\E)^4=0$, then $\E$ is not big by \cite[Rmk.~2.2]{lm}. 

Next, to consider the general case in which $c_2(\E)^2=0$. Since the case $c_2(\E)=0$ has been already analyzed in \cite[Thm.~2]{blv}, we suppose that $c_2(\E) \ne 0$. As we will see below, Ulrich bundles with $c_2(\E)^2=0, c_2(\E) \ne 0$ do not occur if $\rho(X)=1$ and give several restrictions on the geometry of $X$. When $c_3(\E) \ne 0$, we can actually classify them as follows.

\begin{bcor} \hskip 3cm 
\label{c22=0}

Let $X \subset \P^N$ be a smooth irreducible variety of dimension $n \ge 4$ and let $\E$ be a rank $r$ Ulrich bundle on $X$. Assume that $c_2(\E) \ne 0$ and $c_2(\E)^2=0$. Then $\rho(X) \ge 2$ and one of the following occurs:
\begin{itemize}
\item [(i)] $\det \E$ is not $(n-4)$-ample, $\E$ is not $(n-4)$-ample and $X$ contains a linear space $M \subset \P^N$ with $\dim M=n-3$ such that $\E_{|M}$ is trivial.
\item [(ii)] $\det \E$ is $(n-4)$-ample, $c_i(\E)=0$ for $i \ge 3$ and $\E$ is not $(n-4)$-ample if $r \ge 3$. Moreover, $X$ has a morphism with connected fibers onto a smooth curve, and a general Ulrich subvariety associated to $\E$ is a complete intersection of two divisors on $X$.
\end{itemize}
Also, if $c_3(\E) \ne 0$, we have:
\begin{itemize}
\item [(i1)] If $n \ge 5$, $(X,\O_X(1),\E)$ is a linear Ulrich triple over a smooth threefold.
\item [(i2)] If $n= 4$, $(X,\O_X(1),\E)$ is as in one of the cases (ii1), (iv), (v1), (v3) and (vi)-(x) of \cite[Thm.~1]{lms1} with $r \ge 3$.
\end{itemize}
\end{bcor}

For an example that is not linear Ulrich triple but $X$ has a morphism with connected fibers onto a smooth curve, see \cite[Ex.~10.1]{blv} (or Lemma \ref{pf}). On the other hand, in Example \ref{esbs}, we have that $(X,\O_X(1))$ is not $\P^k$-bundle over another variety, but there are several linear spaces $M \subset X$ as in Corollary \ref{c22=0}. 

Next, we consider the case of an Ulrich bundle with $c_3(\E)=0$ (this case always occurs when $c_2(\E)^2=0, c_1(\E)^4 \ne 0$ by Lemma \ref{c3no}(ii)). Except for case (a) below, we do not have a classification result here, but we can show that the pair $(X, \O_X(1))$ belongs to some well-known classes of varieties covered by lines.

\begin{bthm} \hskip 3cm
\label{trecasi}

Let $X \subset \P^N$ be a smooth irreducible variety of dimension $n \ge 3$, degree $d \ge 2$ and let $\E$ be a rank $r \ge 3$ Ulrich bundle on $X$. If $c_3(\E)=0$, then $(X, \O_X(1))$ is one of the following:
\begin{itemize}
\item[(a)] A linear $\P^{n-1}$-bundle over a smooth curve.
\item[(b)] A linear $\P^{n-2}$-bundle over a smooth surface.
\item[(c)] A hyperquadric fibration over a smooth curve. 
\end{itemize}
If $c_2(\E)^2=c_3(\E)=0$ and $n \ge 4$, we have: 
\begin{itemize}
\item[(a1)] In case (a), either $(X,\O_X(1),\E)$ is a linear Ulrich triple over the curve or $\E$ is as in Lemma \ref{pf}(i).
\item[(b1)] In case (b), either $(X,\O_X(1),\E)$ is a linear Ulrich triple over the surface or $n=4$ and $\det \E$ is not ample. 
\item[(c1)] In case (c), $n=4, B \cong \P^1$ and $\det \E$ is not ample.
\end{itemize}
\end{bthm}

Putting together the above results, we have the following almost complete classification of Ulrich bundles with $c_2(\E)^2=0$.

\begin{bthm} \hskip 3cm
\label{c2^2=0}

Let $X \subset \P^N$ be a smooth irreducible variety of dimension $n \ge 4$, degree $d \ge 2$ and let $\E$ be a rank $r \ge 2$ Ulrich bundle on $X$ with $c_2(\E)^2=0$. Then we have:
\begin{itemize}
\item [(i)] $c_2(\E)=0$ if and only if $(X,\O_X(1),\E)$ is a linear Ulrich triple over a curve.
\end{itemize}
Now assume that $c_2(\E) \ne 0$. Then:
\begin{itemize}
\item[(ii)] If $c_1(\E)^4=0$ then $(X,\O_X(1),\E)$ is as in Theorem \ref{c1^4=0}(ii) and (iv).
\item[(iii)] If $c_1(\E)^4 \ne 0$ and $c_3(\E) \ne 0$, then $(X,\O_X(1),\E)$ is as in Corollary \ref{c22=0}(i1)-(i2).
\item[(iv)] If $c_1(\E)^4 \ne 0, c_3(\E)=0$ and $r \ge 3$, then either $(X,\O_X(1),\E)$ is as in Lemma \ref{pf}(i), or is a linear Ulrich triple over a smooth surface, unless possibly $n=4$ and $\det \E$ is not ample. 
\end{itemize}
\end{bthm}

We point out that, while a triple $(X,\O_X(1),\E)$ is as in Lemma \ref{pf} exists even when $r=2, 2 \le n \le 4$ (see \cite[Ex.~10.1]{blv}), we have no examples in the two cases excluded in (iv) above.

We now consider the geometry of Ulrich subvarieties associated to $\E$ (see Definition \ref{usv}).

First, Theorem \ref{trecasi} allows to show that Ulrich subvarieties are often irreducible.

\begin{bcor} \hskip 3cm
\label{conn}

Let $X \subset \P^N$ be a smooth irreducible variety of dimension $n \ge 3$ and let $\E$ be a rank $r \ge 3$ Ulrich bundle on $X$ with $c_2(\E) \ne 0$. Then all Ulrich subvarieties associated to $\E$ are irreducible under one of the following hypotheses: 
\begin{itemize}
\item[(1)] $(X, \O_X(1))$ is not one of the following:
\begin{itemize}
\item[(a)] A linear $\P^{n-1}$-bundle over a smooth curve.
\item[(b)] A linear $\P^{n-2}$-bundle over a smooth surface.
\item[(c)] A hyperquadric fibration over a smooth curve. 
\end{itemize}
\end{itemize}
\begin{itemize}
\item[(2)] $c_2(\E)^2=0, n \ge 4$ and
\begin{itemize}
\item[(a1)] $(X,\O_X(1))$ is as in $(a)$ and $\E$ is not as in Lemma \ref{pf}(i).
\item[(b1)] $(X,\O_X(1))$ is as in $(b)$, $(X,\O_X(1),\E)$ is not a linear Ulrich triple over the surface and either $n \ge 5$ or $n=4$ and $\det \E$ is ample. 
\item[(c1)] $(X, \O_X(1))$ is as in (c) and either $n \ge 5$ or $n=4$ and either $B \not\cong \P^1$ or $B \cong \P^1$ and $\det \E$ is ample.
\end{itemize}
\end{itemize}
\end{bcor}

When $(X, \O_X(1))$ is of type (a)-(c) above, Ulrich subvarieties might be both reducible and irreducible, see Examples \ref{a}, \ref{b}, \ref{c}. On the other hand, when $c_2(\E)^2=0$, we can give a description of reducible (or equivalently disconnected, since they are normal) general Ulrich subvarieties, as follows.

\begin{bcor} \hskip 3cm
\label{conn2}

Let $X \subset \P^N$ be a smooth irreducible variety of dimension $n \ge 4$ and let $\E$ be a rank $r$ Ulrich bundle on $X$ with $c_2(\E) \ne 0$ and $c_2(\E)^2=0$ and let $Z$ be a general Ulrich subvariety associated to $\E$. Assume that $(X, \O_X(1))$ is as in (a)-(c) of Corollary \ref{conn} and, if $n=4$, in case (b), $\det \E$ is ample, in case (c), either $B \not\cong \P^1$ or $B \cong \P^1$ and $\det \E$ is ample (in particular this assumption holds if $Z$ is disconnected, $r \ge 3$ and $(n, B, \det \E)$ satisfy the above). Then, either
\begin{itemize}
\item[(i)] $(X, \O_X(1), \E) \cong (\P^s \times \P^3, \O_{\P^s}(1) \boxtimes  \O_{\P^3}(1), \pi_2^*(\O_{\P^3}(1))^{\oplus r})$, $s \in \{1, 2\}$ and $Z$ is a $\P^s$-bundle over a smooth irreducible curve $Z_B \subset \P^3$ defined by the $(r-2) \times (r-2)$'s minors of an $r \times (r-2)$ matrix of general linear forms, or
\end{itemize}
\begin{itemize}
\item[(ii)] $(X, \O_X(1))$ is as in (a)-(b) and $Z$ is a disjoint union of $\P^{n-2}$'s.
\end{itemize}
Moreover, if $r \ge 3$, any $Z$ in (ii) is disconnected.
\end{bcor}

\section{Notation}

Unless otherwise specified, we henceforth establish throughout the paper the following:

\begin{notation} 

\hskip 3cm

\begin{itemize} 
\item $X \subset \P^N$ is a smooth irreducible variety of dimension $n \ge 1$ and degree $d$.
\item $H \in |\O_X(1)|$ is a very ample divisor. 
\item $\rho(X)$ is the Picard number of $X$.
\item $\nu(\L)$ is the numerical dimension of a nef line bundle $\L$ on $X$.
\item $\P(\F) = \Proj(\Sym(\F))$, where $\F$ is a vector bundle on $X$.
\item For $1 \le i \le n-1$, let $H_i \in |H|$ be general divisors and set $X_n:=X$ and $X_i=H_1\cap \ldots \cap H_{n-i}$. 
\item For $n \ge 2$, we let $Q_n \subset \P^{n+1}$ be a smooth quadric and we let $\mathcal S, \mathcal S', \mathcal S''$ be the spinor Ulrich bundles on $Q_n$. 
\item We denote by $\mathbb G(k, \P(V))$ the Grassmannian of $k$-dimensional projective subspaces of a projective space $\P(V)$.
\end{itemize} 
\end{notation}

We work over the complex numbers.

\section{Definitions and preliminary results}

\subsection{$q$-ample bundles}

We recall, and we will often use, the following definition and fact.

\begin{defi}
Let $q \ge 0$ and let $\L$ be a line bundle on $X$. We say that $\L$ is {\it $q$-ample} if for every coherent sheaf $\F$ on $X$, there exists an integer $m_0>0$ such that $H^i(\F(m\L))=0$ for $m \ge m_0$ and $i > q$. Let $\E$ be a vector bundle on $X$. We say that $\E$ is {\it $q$-ample} if $\O_{\P(\E)}(1)$ is $q$-ample. 
\end{defi}

\begin{remark}
\label{somm} 
Let $\E$ be globally generated bundle. It follows from \cite[Prop.~1.7]{s} that $\E$ is $q$-ample if and only if it is $q$-ample in the sense of \cite[\S 2.1, page 43]{bs}, \cite[Def.~1.3]{s}, that is if every fiber of the morphism associated to $\O_{\P(\E)}(1)$ is at most $q$-dimensional. Moreover, if $\E$ is $q$-ample, then $\det \E$ is $q$-ample by \cite[Cor.~1.10]{s}.\end{remark}

\subsection{Varieties covered by lines}
\label{cov}

We study here some families of varieties covered by lines. Even though this subject has been widely studied, we could not locate in the literature some of the ensuing results.

A tool that we will use is the Fano variety of lines in $X$.

\begin{defi}
Let $X \subset \P^N$ and let $x \in X$ be a point. We denote by $F_1(X, x)$ the variety of lines $L \subset \P^N$ such that $x \in L \subset X$. We say that {\it $X$ is covered by lines} if $F_1(X, x) \ne \emptyset$ for a general point $x \in X$.
\end{defi}

\begin{remark}
When $X$ is covered by lines, then $F_1(X,x)$ is smooth on a general $x$ and $\dim_{[L]} F_1(X,x) = - K_X \cdot L -2$ for every $[L] \in F_1(X,x)$ (see for example \cite[Prop.s~2.2.1 and 2.3.9]{ru}).
\end{remark}

Among varieties covered by lines, we have the following (see for example \cite[Def.~page 461]{i}, \cite{lp}).

\begin{defi}
\label{fib2}
Let $X \subset \P^N$ be a smooth irreducible variety of dimension $n \ge 2$. 

\noindent $\bullet$ We say that $(X, \O_X(1))$ is a {\it linear $\P^k$-bundle} over a smooth variety $B$ if $(X, \O_X(1))=(\P(\F), \O_{\P(\F)}(1))$, where $\F$ is a very ample vector bundle on $B$ of rank $k+1$.

\noindent $\bullet$ We say that $(X, \O_X(1))$ is a {\it hyperquadric fibration} over a smooth curve $B$, if there exists a surjective morphism $f: X \to B$ such that every closed fiber is isomorphic to a quadric $Q \subset \P^n$ and $\O_X(1)_{|Q} \cong \O_Q(1)$.
\end{defi}

Consider now the following classes of pairs $(X, \O_X(1))$, with $X$ covered by lines:

\begin{itemize}
\item[(a)] A linear $\P^{n-1}$-bundle over a smooth curve.
\item[(b)] A linear $\P^{n-2}$-bundle over a smooth surface.
\item[(c)] A hyperquadric fibration over a smooth curve. 
\item[(d)] A Del Pezzo $n$-fold (that is $-K_X=(n-1)H$), except $(\P^3, \O_{\P^3}(2))$.
\end{itemize}

We have

\begin{prop} 
\label{abc}
Let $X \subset \P^N$ be a smooth irreducible variety of dimension $n \ge 3$, degree $d \ge 3$ and such that $X_3$ is covered by lines. Then $(X, \O_X(1))$ is one of (a)-(d). 
\end{prop}
\begin{proof}
Since $X_3$ is covered by lines, it follows from \cite[Thm.~1.4]{lp} (for (b) use also \cite[Thm.~0.2]{sv}) that $(X_3, \O_{X_3}(1))$ is one of (a)-(d) with $n=3$. Moreover, if we are not in case (a), \cite[Thm.~4.1]{lp} also gives that $\dim F_1(X_3, x)=0$ for a general point $x \in X_3$.

We will prove inductively that if $(X_3, \O_{X_3}(1))$ is of one type in $\{(a), (b), (c), (d)\}$, then so is $(X_i, \O_{X_i}(1))$ for all $4 \le i \le n$.

Assume that $(X_3, \O_{X_3}(1))$ is a Del Pezzo $3$-fold. For $i \ge 4$, if $(X_{i-1}, \O_{X_{i-1}}(1))$ a Del Pezzo $(i-1)$-fold, we have that $(K_{X_i} + (i-1)H_{|X_i})_{|X_{i-1}}=K_{X_{i-1}} + (i-2)H_{|X_{i-1}}=0$. Since $\Pic(X_i) \to \Pic(X_{i-1})$ is an isomorphism by Lefschetz's theorem, we find that $K_{X_i} + (i-1)H_{|X_i}=0$, that is $(X_i, \O_{X_i}(1))$ is a Del Pezzo $i$-fold. 

Therefore, we can suppose that $(X_3, \O_{X_3}(1))$ is one of (a)-(c) with $n=3$, that is there is a morphism $p_3 : X_3 \to B$ as in (a)-(c). We now claim that, for each $i$ such that $3 \le i \le n$, there is a morphism $p_i : X_i \to B$ so that $(X_i, \O_{X_i}(1))$ is as in (a) (respectively (b); respectively (c)) as long as $(X_3, \O_{X_3}(1))$ is as in (a) (respectively (b); respectively (c)) and, if $i \ge 4$, $p_{i-1}={p_i}_{|X_{i-1}}$. 

We first assume that $(X_3, \O_{X_3}(1))$ is as in (a) or (c), so that $\dim B=1$, and we prove the claim by induction on $i$. The claim being trivially true for $i=3$, suppose that $i \ge 4$ and that there is a morphism $p_{i-1} : X_{i-1} \to B$ so that $(X_{i-1}, \O_{X_{i-1}}(1))$ is as in (a) or (c) and, if $i \ge 5$, $p_{i-2}={p_{i-1}}_{|X_{i-2}}$. Note that $i-1-\dim B \ge 2$, so that $p_{i-1} : X_{i-1} \to B$ extends to some morphism $p_i : X_i \to B$ by \cite[Thm.~5.2.1]{bs}. Now, in the embedding $X_i \subset \P H^0(\O_{X_i}(1)) = \P^M$ we have that $X_{i-1}$ is a hyperplane section of $X_i$ and each fiber $F_{i-1}$ of $p_{i-1}$ is a linear (possibly degenerate) $\P^{i-2} \subset \P^M$ in case (a) or a linear (possibly degenerate) $(i-2)$-dimensional quadric $Q \subset \P^M$ in case (c). Moreover, since $p_i$ extends $p_{i-1}$, each fiber $F_i$ of $p_i$ is a variety in $\P^M$ whose hyperplane section is $F_{i-1}$. Therefore also $F_i$ is a linear (possibly degenerate) $\P^{i-1} \subset \P^M$ in case (a) or a linear (possibly degenerate) $(i-1)$-dimensional quadric $Q' \subset \P^M$ in case (c). Since $\O_{X_i}(1)_{|X_{i-1}} = \O_{X_{i-1}}(1)$, it follows that $(X_i, \O_{X_i}(1))$ is as in (a) or (c).

Now assume that $(X_3, \O_{X_3}(1))$ is as in (b) but not as in (a), (c), so that $\dim B=2$, and we prove the claim by induction on $i$. The claim being trivially true for $i=3$, suppose that $i \ge 4$ and that there is a morphism $p_{i-1} : X_{i-1} \to B$ so that $(X_{i-1}, \O_{X_{i-1}}(1))$ is as in (b) but not as in (a), (c) and, if $i \ge 5$, $p_{i-2}={p_{i-1}}_{|X_{i-2}}$. Note that $i-1-\dim B = i-3$. We have that \cite[Conj.~5.5.1]{bs} is true: if $i \ge 5$ by \cite[Thm.~5.5.2]{bs}, while, if $i=4$, by what is explained on \cite[page 118]{bs}. We are certainly done by \cite[Conj.~5.5.1]{bs} if $X_{i-1} \not\cong \P^1 \times \P^{i-2}$ or if $X_{i-1} \cong \P^1 \times \P^{i-2}$ but $p_{i-1} : X_{i-1} \to B$ is not the second projection. It remains to study the case in which $X_{i-1} \cong \P^1 \times \P^{i-2}$ and $p_{i-1}=\pi_2 : X_{i-1} \to B=\P^{i-2}$ is the second projection, so that $i=4$ and $X_3 \cong \P^1 \times \P^2$. We will prove that this case does not occur.

Since \cite[Conj.~5.5.1]{bs} is true, if we set $p'_3 = \pi_1$ the first projection, we have that $(X_4, \O_{X_4}(1))$ is as in (a) over $\P^1$ and the bundle map $p_4' : X_4 \to \P^1$ extends $p'_3$. Now, in the embedding $X_4 \subset \P H^0(\O_{X_4}(1)) = \P^M$ we have that $X_3$ is a hyperplane section of $X_4$ and each fiber of $p_4'$ is a linear $\P^3 \subset \P^M$. Moreover, since $p_4'$ extends $p'_3$, each fiber $F'=\{z\} \times \P^2$ of $p'_3$ is isomorphic to $\P^2$ and is a hyperplane section of a fiber of $p_4'$, that is of a linear $\P^3$, hence each fiber of $p'_3$ is a linear $\P^2$ in $\P^M$. Since, by definition, $\O_{X_3}(1) = \O_{X_4}(1)_{|X_3}$, we deduce that $\O_{X_3}(1)_{|F'} \cong \O_{\P^2}(1)$. Let $H_1= \pi_1^*\O_{\P^1}(1), H_2= \pi_2^*\O_{\P^2}(1)$. If $G \in |\O_{X_3}(1)|$, we can write $G=uH_1+vH_2$ for some positive integers $u, v$. Now, $\O_{\P^2}(1) \cong \O_{X_3}(1)_{|F'} \cong \O_{X_3}(uH_1+vH_2)_{|\{z\} \times \P^2} \cong \O_{\P^2}(v)$, so that $v=1$. Let $x=(x_1, x_2) \in X_3 \cong \P^1 \times \P^2$ be a general point and let $[R] \in F_1(\P^2,x_2)$. Then
$$(\{x_1\} \times R) \cdot G = (\{x_1\} \times R) \cdot (uH_1+H_2)=1$$
so that $[\{x_1\} \times R] \in F_1(X_3, x)$. Therefore, $F_1(\P^2, x_2) \subset F_1(X_3,x)$, and we get the contradiction $\dim F_1(X_3, x)) \ge 1$. This proves the claim and concludes the proof.
\end{proof}

We also observe that (a)-(d) cannot mix when $n \ge 4$, with one well-known exception.

\begin{lemma}
\label{unasola}
Let $X \subset \P^N$ be a smooth irreducible variety of dimension $n \ge 4$. Then $(X, \O_X(1))$ can be only one of (a), (b), (c) in Proposition \ref{abc}. Moreover, if $(X, \O_X(1))$ is as in (d), then it is not as in (a) or (c), while it is as in (b) if and only if 
$(X, \O_X(1))=(\P^2 \times \P^2, \O_{\P^2}(1) \boxtimes \O_{\P^2}(1))$.
\end{lemma}
\begin{proof}
Assume that $(X, \O_X(1))$ is as in (a) and as in (b) (or (c)). Let $p : X \to B$ be the morphism giving the $\P^{n-2}$-bundle over a smooth surface in case (b) or the hyperquadric fibration over a smooth curve in case (c). Let $x \in X$ be a general point and let $x \in \P^{n-1} \subset X$. Then we have that the morphism $p_{|\P^{n-1}} : \P^{n-1} \to B$ must be constant because $\dim B \le 2 < n-1$. Therefore $\P^{n-1}$ is contained in a fiber of $p$, a contradiction both in case (b) and in case (c), since a general fiber is a smooth irreducible quadric.

Now, assume that $(X, \O_X(1))$ is as in (b) and (c). Let $p : X \to B$ be the morphism giving the $\P^{n-2}$-bundle over a smooth surface, let $x \in X$ be a general point and let $x \in Q_{n-1} \subset X$. Then we have that the morphism $p_{|Q_{n-1}} : Q_{n-1} \to B$ must be constant because $\dim B=2 < n-1$. Therefore $Q_{n-1}$ is contained in a fiber of $p$, a contradiction.

Finally, if $(X, \O_X(1))$ is as in (d) and one of (a), (b) or (c), then $\rho(X) \ge 2$, hence the classification of Del Pezzo manifolds (see for example \cite[pages 860-861]{lp}, \cite[Table, page 710]{f1}) implies that $(X,\O_X(1)) = (\P^2 \times \P^2, \O_{\P^2}(1) \boxtimes \O_{\P^2}(1))$.
\end{proof}

\subsection{Fivefolds and adjunction theory}
\label{adj}

We now collect some definitions and standard facts in adjunction theory, that we recall for completeness' sake. 

\begin{defi}
\label{fib}
Let $X \subset \P^N$ and assume that there exists a surjective morphism with connected fibers $\phi: X \to B$, over a normal variety $B$ of dimension $m$ such that $K_X+(n-m+q)H=\phi^*\L$ and $\L$ is ample on $B$. We say that $(X, \O_X(1))$ is a {\it scroll over $B$} if $q=1$, {\it a quadric fibration over $B$} if $q=0$, {\it a Del Pezzo fibration over $B$} if $q=-1$. \end{defi}

It is well known that, when $n \ge 4$, the definition of quadric fibration coincides with the one given in Definition \ref{fib2}, see for example \cite[Lemma 1 and Prop.~2]{la}.

If $K_X$ is not nef, consider the nef value of $(X,H)$ (see \cite[Def.~1.5.3]{bs})
$$\tau = \tau(X,H) = \min\{t \in \R : K_X+tH \ \hbox{is nef}\}$$
and the nef value morphism, defined for $m \gg 0$ by
$$\phi_{\tau} =  \phi_{\tau}(X,H):= \varphi_{m(K_X+\tau H)} : X \to X'.$$ 
We recall that $(\phi_{\tau})_*\O_X \cong \O_{X'}$, see \cite[Def.~1.5.3]{bs}.

In dimension $5$ we have.

\begin{lemma}
\label{aggiu}
Let $X \subset \P^N$ be a smooth irreducible fivefold such that $\dim F_1(X,x) \ge 1$ on a general point $x \in X$. Let $H \in |\O_X(1)|$, let $\tau$ be the nef value of $(X,H)$ and let $\phi_{\tau}$ be the nef value morphism. Then $(X,\O_X(1))$ is only one of the following:
\begin{itemize}
\item [(a)] $(\P^5, \O_{\P^5}(1))$. 
\item [(b.1)] $(Q_5, \O_{Q_5}(1))$.
\item [(b.2)] A linear $\P^4$-bundle over a smooth curve.
\item [(c.1)] A Del Pezzo $5$-fold, that is $K_X=-4H$.
\item [(c.2)] A quadric fibration under $\phi_{\tau} : X \to X'$ over a smooth curve.
\item [(c.3)] A linear $\P^3$-bundle over a smooth surface.
\item [(d.1)] A Mukai variety, that is $K_X=-3H$.
\item [(d.2)] A Del Pezzo fibration under $\phi_{\tau} : X \to X'$ over a smooth curve with $\tau=3$.
\item [(d.3)] A quadric fibration with equidimensional fibers under $\phi_{\tau} : X \to X'$ over a smooth surface.
\item [(d.4)] A linear $\P^2$-bundle under $\phi_{\tau} : X \to X'$ over a smooth threefold with $\tau=3$.
\item [(e)] The blow-up of a smooth fivefold at $t \ge 1$ points, with exceptional divisors $E_i \cong \P^4$ such that $H_{|E_i} \cong \O_{\P^4}(1), 1 \le i \le t$.
\end{itemize}
Moreover, in cases (d.1)-(d.4) we have that $\phi_{\tau}$ contracts any line $L$ such that $x \in L \subset X$.
\end{lemma}
\begin{proof}
Let $x \in X$ be a general point and let $[L] \in F_1(X,x)$. By hypothesis we have
$$1 \le \dim_{[L]} F_1(X,x) = -K_X \cdot L - 2$$
so that $K_X \cdot L \le - 3$. Hence $K_X$ is not nef and then $\tau \ge 3$. Also, if $\tau = 3$, then $K_X \cdot L= - 3$, hence $m(K_X+3H) \cdot L = 0$ for every $m \ge 1$. Therefore $K_X+3H$ is nef and not big, since the associated morphism contracts $L$ to a point. In particular also $\phi_{\tau}$ contracts $L$ to a point. This proves the last assertion of the lemma. 

We now continue the proof.

By \cite[Prop.~7.2.2]{bs} we have that either we are in case (a), or $\tau=5$ and we are in cases (b.1) or (b.2) or $\tau \le 5$ and $K_X+5H$ is big and nef. In the latter case $K_X+5H$ is ample by \cite[Prop.~7.2.3]{bs} and $\tau \le 4$ by \cite[Prop.~7.2.4]{bs}. Moreover \cite[Prop.~7.3.2]{bs} gives that $K_X+4H$ is ample unless $\tau = 4$ and either we are in one of the cases (c.1)-(c.3) or (e) (for (c.3) use also \cite[Thm.~0.2]{sv} and for (e) use also \cite[Thm.~0.3 and Rmk.~1]{sv}). Next if $K_X+4H$ is ample then $\tau < 4$ and $(X,\O_X(1))$ is isomorphic to its first reduction (see \cite[Def. 7.3.3]{bs}). Therefore \cite[Prop.~7.3.4]{bs} implies that $\tau = 3$, so that, as above, $K_X+3H$ is nef and not big. It follows from \cite[Prop.~7.5.3 and Thm.~14.2.3]{bs} that either we are in one of the cases (d.1)-(d.3) or $(X,\O_X(1))$ is a scroll under $\phi_{\tau} : X \to X'$ over a normal threefold. Finally, in the latter case, we are in case (d.4) by \cite[Prop.~14.1.3 and Prop.~3.2.1]{bs}.
\end{proof}

\section{Generalities on globally generated bundles}

We collect here some notation and results about globally generated bundles, that will be used throughout.

\begin{defi}
\label{notgg}
Let $\E$ be a rank $r$ globally generated vector bundle on $X$. We define the map determined by $\E$ as
$$\Phi=\Phi_{\E} : X \to {\mathbb G}(r-1, \P H^0(\E)).$$
For any point $x \in X$, we will denote the fiber of $\Phi$ by
$$F_x = \Phi^{-1}(\Phi(x))$$
and we set $\phi(\E)$ for the dimension of the general fiber of $\Phi_{\E}$.  
\end{defi}
We recall that, considering the map $\lambda_{\E} : \Lambda^r H^0(\E) \to H^0(\det \E)$, one gets a commutative diagram
\begin{equation}
\label{muk}
\xymatrix{X \ar[d]^{\varphi_{|\Im \lambda_{\E}|}} \ar[r]^{\hskip -1.2cm \Phi_{\E}} &  {\mathbb G}(r-1,\P H^0(\E)) \ar@{^{(}->}[d]^{P_{\E}} \\ \P \Im \lambda_{\E} \ar@{^{(}->}[r] & \P \Lambda^r H^0(\E)}
\end{equation}
where $P_{\E}$ is the Pl\"ucker embedding. In particular this implies that 
\begin{equation}
\label{muk2}
\phi(\E)=n-\nu(\det \E)
\end{equation}
where $\nu(\det \E)$ is the numerical dimension of $\det \E$.

We will also use degeneracy loci.

\begin{defi}
Let $\varphi : \E \to \F$ be a morphism of vector bundles of ranks $e, f$ on $X$. For any $k \in \Z$ with $0 \le k \le \min\{e,f\}$ we denote the $k$-th degeneracy locus of $\varphi$ by $D_k(\varphi)$, that is
$$D_k(\varphi)=\{x \in X : \rk \varphi(x) \le k\}.$$
\end{defi}

\begin{remark}
\label{degeneracy}
Degeneracy loci have a natural scheme structure, given locally by the vanishings of the $(k+1) \times (k+1)$ minors of the matrix defining $\varphi(x)$. Equivalently \cite[\S (2.1)]{las}, the ideal sheaf of $D_k(\varphi)$ is the image of the morphism $\Lambda^{k+1} \E \otimes \Lambda^{k+1} \F^* \to \O_X$ induced by $\Lambda^{k+1} \varphi$.
\end{remark}

We observe that, in the cases we are interested about, degeneracy loci move on $X$.

\begin{lemma}
\label{mov}
Let $\E$ be a globally generated rank $r$ vector bundle on $X$ with $c_k(\E) \ne 0$, for some $k \ge 1$. Let $x \in X$ be a general point. Suppose that either $k=r$ or $k \le r-1$ and $H^1(\O_X)=0$. Then there is a morphism $\varphi :  \O_X^{\oplus (r+1-k)} \to \E$ such that $D_{r-k}(\varphi)$ is reduced of pure codimension $k$ and $x \in D_{r-k}(\varphi)$.
\end{lemma} 
\begin{proof}
Since $c_k(\E) \ne 0$, we have that $k \le \min\{r,n\}$. When $k=r$ we will call it the top case. Suppose first that we are in the top case and consider the incidence correspondence 
$$\I = \{(y, [\sigma]) \in X \times \P H^0(\E): y \in Z(\sigma)\} \subset X \times \P H^0(\E).$$
Set $t=h^0(\E)$. Since $\E$ is globally generated, we have that $t \ge r+1$, for otherwise $\E \cong \O_X^{\oplus r}$, contradicting $c_r(\E) \ne 0$. Also, for every $y \in X$, we get that $h^0(\I_{\{y\}/X} \otimes \E)=t-r$ and every fiber of the first projection $\pi_1 : \I \to X$ is isomorphic to $\P H^0(\I_{\{y\}/X} \otimes \E)=\P^{t-r-1}$. Hence $\pi_1$ is surjective and $\I$ is irreducible of dimension $n+t-r-1$. Now, the second projection $\pi_2 : \I \to \P H^0(\E) = \P^{t-1}$ has general fibers isomorphic to $Z(\sigma)$, which is reduced of pure codimension $r$ by \cite[Lemmas 4.1 and 3.5]{blv}. Therefore also $\pi_2$ is surjective and it follows that we can find a general $\sigma \in H^0(\E)$ with $x \in Z(\sigma)=D_0(\varphi)$, where $\varphi :  \O_X^{\oplus (r+1-k)} \to \E$ is the morphism associated to $\sigma$. Thus, the top case is proved.

Now assume that $k \le r-1$ and let $k'= \min\{l \ge k : c_{l+1}(\E)=0\}$, so that $c_{k'}(\E) \ne 0, c_{k'+1}(\E)=0$ and $k' \le r$. Let $V' \subset H^0(\E)$ be a general subspace with $\dim V'=r+1-k'$. Choose a general subspace $V$ with $V' \subset V \subset H^0(\E)$ and $\dim V=r+1-k$. Then we have morphisms $\varphi' : V' \otimes \O_X \to \E, \varphi : V \otimes \O_X \to \E$, and clearly $D_{r-k'}(\varphi') \subset D_{r-k}(\varphi)$. Hence we will be done if we prove that
\begin{equation}
\label{mob}
\hbox{for a general subspace} \ V' \ \hbox{as above, we have that} \ x \in D_{r-k'}(\varphi'). 
\end{equation}
In fact, \eqref{mob} implies that $D_{r-k}(\varphi)$ is nonempty, hence reduced of pure codimension $k$ by \cite[Statement (folklore)(i), \S 4.1]{ba} and $x \in D_{r-k}(\varphi)$. 

We now prove \eqref{mob}. If $k'=r$ we are done by the top case. Next, assume that $k' \le r-1$. Choose a general subspace $V'' \subset H^0(\E)$ with $\dim V''=r-k'$. We get a morphism $\varphi'' : V'' \otimes \O_X \to \E$ and an exact sequence
$$0 \to V'' \otimes \O_X \to \E \to \F \to 0$$
where $\F$ is a rank $k'$ vector bundle on $X$, since $c_{k'+1}(\E)=0$, hence $D_{r-k'-1}(\varphi'')= \emptyset$ by \cite[Lemma 4.1]{blv}. Moreover, since $H^1(\O_X)=0$, we have an exact sequence
$$0 \to V'' \to H^0(\E) \mapright{\alpha} H^0(\F) \to 0.$$
Now, $c_{k'}(\F)=c_{k'}(\E) \ne 0$, hence, by the top case applied to $\F$, we can find a general $\tau \in H^0(\F)$ with $x \in Z(\tau)$. Now let $\sigma \in H^0(\E)$ be such that $\alpha(\sigma)=\tau$
and let $V'=\langle V'', \sigma \rangle \subset H^0(\E)$. We have by construction that $Z(\tau)=D_{r-k'}(\varphi')$  and we are done.
\end{proof}

In the sequel, we will be interested in vanishing of some Chern classes. We first show the following, probably well-known, fact.

\begin{lemma}
\label{eff}
Let $\E$ be a globally generated vector bundle on $X$. Then the classes 
$\prod_{i=1}^kc_i(\E)^{m_i}$
are effective for all integers $k \ge 1, m_i \ge 1, 1 \le i \le k$.
\end{lemma}
\begin{proof}
We argue by induction on $m:= \sum_{i=1}^k m_i$. If $m=1$, the claim is known. For the inductive step, we consider the class $c_j(\E) \cdot \prod_{i=1}^kc_i(\E)^{m_i}$ for some integer $j \ge 1$ and we can assume that 
$\alpha:= \prod_{i=1}^kc_i(\E)^{m_i}$ is effective. If $\alpha=0$, then $c_j(\E)\alpha=0$ is effective. Otherwise, we have that $\alpha=\sum_{h=1}^t a_h[W_h]$, where $a_h \ge 1$ and each $W_h$ is an integral codimension $m$ subvariety. Since $c_j(\E) \cdot \alpha=\sum_{h=1}^t a_h c_j(\E) \cdot [W_h]$, it is enough to show that $c_j(\E) \cdot [W_h]$ is effective for each $h$. To this end, set $W=W_h$ and let $i : W \hookrightarrow X$ be the inclusion. Since $i^*\E$ is globally generated, we have by \cite[Ex.~14.3.2(d)]{fu} that $c_j(i^*\E)$ is effective on $W$, hence so is $i_*(c_j(i^*\E))$ on $X$. Now, the projection formula \cite[Thm.~3.2(c)]{fu} gives 
$$i_*(c_j(i^*\E))=i_*(c_j(i^*\E)\cdot [W])=c_j(\E)\cdot i_*([W])=c_j(\E)\cdot [W]$$
hence our claim.
\end{proof}

\begin{lemma}
\label{c3no}
Let $\E$ be a globally generated bundle on $X$. We have:
\begin{itemize}
\item [(i)] If $c_1(\E)^4 = 0$, then $c_1(\E)^2c_2(\E)=c_1(\E)c_3(\E)=c_4(\E)=c_2(\E)^2=0$.
\item [(ii)] If $c_2(\E)^2=0$, then $c_1(\E)c_3(\E)=0$.
Moreover, if $c_3(\E) \ne 0$, then $c_1(\E)^4=c_1(\E)^2c_2(\E)=c_4(\E)=0$.
\end{itemize}
\end{lemma}
\begin{proof}
Let $H$ be a very ample line bundle on $X$ and let $r$ be the rank of $\E$. To prove (i), we can assume that $n \ge 4$, otherwise (i) clearly holds. Restricting to $X_4=H_1\cap \ldots \cap H_{n-4}$, for $H_i \in |H|$ general divisors, we get a globally generated bundle $\E':=\E_{|X_4}$ on $X_4$ with $c_1(\E')^4=c_1(\E_{|X_4})^4=c_1(\E)^4H^{n-4}=0$. It follows from \cite[Cor.~2.7]{dps} that $c_1(\E')^2c_2(\E')=c_1(\E')c_3(\E')=c_2(\E')^2=c_4(\E')=0$. Thus, 
$$c_1(\E)^2c_2(\E)H^{n-4}=c_1(\E)c_3(\E)H^{n-4}=c_2(\E)^2H^{n-4}=c_4(\E)H^{n-4}=0$$
hence (i) follows from Lemma \ref{eff}. Next, to prove (ii), the conclusion being obvious if $n \le 3$, we assume that $n \ge 4$ and that $c_2(\E)^2=0$. With the above notation, we have that $c_2(\E')^2=0$.
Consider the Schur polynomial $s_{(2,2,0,0)}(\E')=c_2(\E')^2-c_1(\E')c_3(\E')=-c_1(\E')c_3(\E')$. 
We have by \cite[Thm.~2.5]{dps} that $s_{(2,2,0,0)}(\E') \ge 0$, hence $c_1(\E')c_3(\E') \le 0$. On the other hand, $c_1(\E')c_3(\E') \ge 0$, since $\det \E'$ is globally generated.
Therefore $c_1(\E')c_3(\E')=0$ and, as in the proof of part (i), we get that $c_1(\E)c_3(\E)=0$.
Now, assume that $c_3(\E) \ne 0$, so that $r \ge 3$ and $c_3(\E')\ne 0$: In fact if $c_3(\E')=0$, then $c_3(\E)H^{n-4}=0$, hence $c_3(\E)=0$ by Lemma \ref{eff}. We first show that there is a rank $3$ globally generated bundle $\G$ on $X'$ such that $c_3(\G)=c_3(\E')$.
In fact, if $r=3$, we just set $\G=\E'$. If $r \ge 4$, let $V' \subset H^0(\E')$ be a general subspace of dimension $r-3$ and let $\varphi' : V' \otimes \O_{X'} \to \E'$ be the associated morphism. Since $c_4(\E')=0$, we have that 
$D_{r-4}(\varphi')=\emptyset$ by \cite[Lemma 4.1]{blv}. Hence $\rk \varphi'(x) = r-3$ for every $x \in X'$ and we get an exact sequence
\begin{equation}
\label{suc1}
0 \to V' \otimes \O_{X'} \to \E' \to \G \to 0
\end{equation}
where $\G$ is a rank $3$ bundle on $X'$. It follows that $\G$ is globally generated and $c_3(\G)=c_3(\E')$. In particular, $c_3(\G) \ne 0$ and Lemma \ref{mov} implies that, given a general point $x \in X'$, we can find a general $\sigma \in H^0(\G)$ with $x \in Z(\sigma)$. On the other hand, we have by \cite[Lemma 4.1]{blv} that $[Z(\sigma)] = c_3(\G) = c_3(\E')$ and therefore $\det \E' \cdot Z(\sigma)=c_1(\E')c_3(\E')=0$. Since $\det \E'$ is globally generated, this means that the curve $Z(\sigma)$ is contracted by the morphism $\varphi_{\det \E'} : X' \to \P H^0(\det \E')$. Therefore, the general fiber of $\varphi_{\det \E'}$ has positive dimension and this implies that $c_1(\E')^4=0$, hence $c_1(\E)^4=0$. Finally, we conclude by applying (i).
\end{proof} 

\begin{remark}
Even in the Ulrich case, Lemma \ref{c3no}(ii) is sharp, in the sense that there are Ulrich bundles $\E$ (hence, in particular, globally generated bundles) with $c_2(\E)^2=c_3(\E)=0$ and $c_1(\E)^n > 0$. See Lemma \ref{pf} or \cite[Ex.~10.1]{blv}.
\end{remark}

\section{Generalities on Ulrich vector bundles}

We will often use the following well-known properties of Ulrich bundles

\begin{lemma}
\label{ulr}
Let $\E$ be a rank $r$ Ulrich bundle on $X \subset \P^N$. We have:
\begin{itemize}
\item[(i)] $\E$ is globally generated and aCM. 
\item[(ii)] $\det \E$ is globally generated and it is not trivial, unless $(X,\O_X(1), \E) = (\P^n, \O_{\P^n}(1), \O_{\P^n}^{\oplus r})$.
\item [(iii)] If $(X,\O_X(1))=(\P^n, \O_{\P^n}(1))$, then $\E=\O_{\P^n}^{\oplus r}$.
\item [(iv)] $\E_{|X_{n-1}}$ is Ulrich on a smooth hyperplane section $X_{n-1}$ of $X$.
\item[(v)] $c_1(\E) H^{n-1}=\frac{r}{2}[K_X+(n+1)H] H^{n-1}$.
\item[(vi)] If $n \ge 2$, then $c_2(\E) H^{n-2}=\frac{1}{2}[c_1(\E)^2-c_1(\E) K_X] H^{n-2}+\frac{r}{12}[K_X^2+c_2(X)-\frac{3n^2+5n+2}{2}H^2] H^{n-2}$.
\item[(vii)] If $n \ge 3$, then

$\begin{aligned}[t] 
c_3(\E) H^{n-3} = & c_1(\E)c_2(\E)H^{n-3}-\frac{1}{3}c_1(\E)^3H^{n-3}+\frac{1}{2}[c_1(\E)^2-2c_2(\E)]K_X H^{n-3} 
\\ & -\frac{1}{6}c_1(\E)[K_X^2+c_2(X)]H^{n-3}+\frac{r}{12}K_Xc_2(X)H^{n-3}+\frac{rn(n+1)^2d}{24}.
\end{aligned}$
\item[(viii)] $\E^*(K_X+(n+1)H)$ is Ulrich on $X$.
\end{itemize} 
\end{lemma}
\begin{proof}
For (i)-(ii) and (v)-(vi), see for example \cite[Lemma 3.1]{lr1}, for (iii)-(iv) see \cite[(3.1), (3.4)]{be}, for (viii) see \cite[Lemma 5.1(ix)]{blv}. As for (vii), we have $K_{X_i}={K_X}_{|X_i}+(n-i)H_{|X_i}$. Moreover, using the exact sequences, for $3 \le i \le n-1$, 
$$0 \to T_{X_i} \to (T_{X_{i+1}})_{|X_i} \to H_{|X_i} \to 0$$
we get by induction that 
$$c_2(X_i)=c_2(X)_{|X_i}+(n-i) {K_X}_{|X_i}H_{|X_i} + \frac{(n-i+1)(n-i)}{2}H^2_{|X_i}$$
hence
\begin{equation}
\label{c2x3}
c_2(X_3)=c_2(X)_{|X_3}+(n-3) {K_X}_{|X_3}H_{|X_3} + \frac{(n-2)(n-3)}{2}H^2_{|X_3}.
\end{equation}
Now (vii) follows from \cite[Prop.~3.7(b)]{bmpt}, \eqref{c2x3}, (v), (vi) and Riemann-Roch $\chi(\O_{X_3})=-\frac{1}{24}K_{X_3}c_2(X_3)$.
\end{proof}

\begin{defi}
\label{not4}
Let $\E$ be a vector bundle on $X \subset \P^N$. We say that $(X,\O_X(1),\E)$ is a {\it linear Ulrich triple} if there are a smooth irreducible variety $B$ of dimension $b \ge 1$, a very ample vector bundle $\F$ and a vector bundle $\G$ on $B$ such that 
\begin{equation}
\label{trip}
(X,\O_X(1),\E)=(\P(\F), \O_{\P(\F)}(1), \pi^*(\G(\det \F)))
\end{equation}
where $\pi: X \cong \P(\F) \to B$ is the projection and 
\begin{equation}
\label{van-trip}
H^j(\G \otimes S^k \F^*)=0 \ \hbox{for all} \ j \ge 0, 0 \le k \le b-1.
\end{equation}
\end{defi}
Note that when $(X,\O_X(1),\E)$ is a linear Ulrich triple, then $\E$ is an Ulrich bundle on $X$ by \cite[Lemma 4.1]{lo}. 

Next, we recall some definitions and facts in \cite{lr1}.

\begin{defi}
\label{usv}
Let $n \ge 2$ and $d \ge 2$. Let $\E$ be a rank $r \ge 2$ Ulrich bundle on $X \subset \P^N$. Let $V \subset H^0(\E)$ be a subspace of dimension $r-1$ such that, if $\varphi : V \otimes \O_X \to \E$ is the associated morphism and $Z=D_{r-2}(\varphi)$, then $Z$ satisfies the following conditions (in particular these hold, by \cite[Statement (folklore)(i), \S 4.1]{ba}, if $V$ is a general subspace of $H^0(\E)$):
\begin{itemize}
\item[(a)] $Z$ is either empty or of pure codimension $2$,
\item[(b)] if $Z \ne \emptyset$ and either $r=2$ or $n \le 5$, then $Z$ is smooth (possibly disconnected),
\item[(c)] if $Z \ne \emptyset$ and $n \ge 6$, then $Z$ is either smooth or is normal, Cohen-Macaulay, reduced and with $\dim \Sing(Z) = n-6$.
\end{itemize}
We call $Z$ {\it an Ulrich subvariety associated to $\E$}. We say that $Z$ {\it is general} if $V$ is a general subspace of $H^0(\E)$.
\end{defi}

\begin{remark}
\label{usv2}
Let $Z \subset X$ be any subvariety satisfying (a)-(c) above and (i)-(vi) of \cite[Thm.~1]{lr1}. It follows from \cite[Thm.~1]{lr1} that there is an Ulrich bundle $\E$ such that $Z$ is associated to $\E$.
\end{remark}

We include two useful lemmas on the $q$-ampleness of the determinant of an Ulrich bundle.

\begin{lemma} 
\label{q}
Let $\E$ be an Ulrich bundle on $X \subset \P^N$. Then $\det \E$ is not $q$-ample if and only if $X$ contains a linear space $M \subset \P^N$ with $\dim M = q+1$ and $\E_{|M}$ is trivial.
\end{lemma}
\begin{proof}
Observe that $\det \E$ is globally generated by Lemma \ref{ulr}(ii). If $X$ contains a linear space $M \subset \P^N$ with $\dim M = q+1$ and $\E_{|M}$ is trivial, then $(\det \E)_{|M}$ is trivial, hence the morphism $\varphi_{\det \E} : X \to \P H^0(\det \E)$ has a fiber $F$ containing $M$. Therefore, $\dim F \ge q+1$ and $\det \E$ is not $q$-ample by Remark \ref{somm}. Vice versa, assume that $\det \E$ is not $q$-ample. Then, again by Remark \ref{somm}, $\varphi_{\det \E}$ has a fiber $M'$ with $\dim M' \ge q+1$. Now $\det(\E_{|M'})=(\det \E)_{|M'}$ is trivial, hence so is $\E_{|M'}$, since it is globally generated. Consider the map $\lambda_{\E} : \Lambda^r H^0(\E) \to H^0(\det \E)$ and the factorization 
$$\xymatrix{X \ar[dr]_{\varphi_{|\Im \lambda_{\E}|}} \ar[r]^{\hskip -.5cm \varphi_{\det \E}} & \varphi_{\det \E}(X) \ar[d]^f& \hskip -1.1cm \subset \P H^0(\det \E) \\ & \varphi_{|\Im \lambda_{\E}|}(X) & \hskip -1.6cm \subset |\Im \lambda_{\E}|}$$
where $f$ is a finite projection. It follows from \cite[Lemma 2.10]{lms1} and \eqref{muk} that the fibers of $\varphi_{|\Im \lambda_{\E}|}$ are linear spaces, hence in particular they are connected. Therefore $f$ is a bijection and the fibers of $\varphi_{|\Im \lambda_{\E}|}$ coincide with the fibers of $\varphi_{\det \E}$. Then, we have that $M'$ is a linear space in $\P^N$ contained in $X$ and choosing a linear space $M \subseteq M'$ with $\dim M = q+1$, we find that $\E_{|M}$ is trivial. 
\end{proof}

\begin{remark}
In the Ulrich case, one can recover the last statement of Remark \ref{somm} in the following way: If $\E$ is $q$-ample, then $\det \E$ is also $q$-ample, for otherwise, by Lemma \ref{q}, $X$ contains a linear space $M \subset \P^N$ with $\dim M = q+1$ and $\E_{|M}$ is trivial, contradicting \cite[Thm.~1]{lr2}.
\end{remark}

\begin{lemma}
\label{qa}
Let $X \subset \P^N$ be a smooth irreducible variety of dimension $n$, let $B$ be a smooth irreducible variety of dimension $m$ and let $\pi : X \to B$ be a flat morphism with connected fibers. Assume that $\Pic(F) \cong \Z \O_F(1)$ for every fiber $F$ of $\pi$. Then $\Pic(X) \cong \Z \O_X(1) \oplus \pi^* \Pic(B)$. 

Also, if $\E$ is an Ulrich bundle on $X$, we have:
\begin{itemize}
\item [(i)] If $\det \E$ is not $m$-ample, then $m \le \frac{n-1}{2}$ and there is a vector bundle $\H$ on $B$ such that $\E \cong \pi^*\H$.
\item [(ii)] If $\det \E$ is not $(m-1)$-ample, then $m \le \frac{n}{2}$ and either there is a vector bundle $\H$ on $B$ such that $\E  \cong \pi^*\H$ or $B \cong \P^m$. Moreover, in the latter case, there is a linear space $M \cong \P^m, M \subset X$ such that $M$ is a $k$-section of $\pi$ with $k \le \deg F$ for every fiber $F$ of $\pi$ and $\E_{|M}$ is trivial.
\end{itemize}
\end{lemma}
\begin{proof}
Since $\pi$ is flat and the fibers $F$ are connected, we have that $\dim F = n-m \ge 1$. Moreover, if $F'$ is another fiber, we have that $\deg F = H^{n-m}F=H^{n-m}F'=\deg F'$. Now let $\L \in \Pic(X)$, so that $\L_{|F} \cong \O_F(l_F)$ for some $l_F \in \Z$. We have
$$l_F \deg F = H^{n-m-1} \L F = H^{n-m-1} \L F' = l_{F'} \deg F'$$
and therefore $l_F=l_{F'}$. Setting $l=l_F$, we get that $(\L - lH)_{|F} \cong \O_F$ for every fiber $F$ and then 
\cite[Prop.~25.1.11]{vak} implies that there is a line bundle $L$ on $B$ such that $\L-lH \cong \pi^*L$. Thus $\Pic(X) \cong \Z \O_X(1)+\pi^* \Pic(B)$. Also, if $\O_X(a) \in \pi^* \Pic(B)$, then $\O_X(a)_{|F} \cong \O_F$ and therefore $a=0$. This proves the first statement of the lemma. Now assume that $\det \E$ is not $m$-ample. Then Lemma \ref{q} gives a linear space $M \cong \P^{m+1}, M \subset X$ such that $\E_{|M}$ is trivial. Since $\pi_{|M} : M \to B$ must be constant, there is a fiber $F$ such that $M \subseteq F$. In particular we have that $m+1 \le n-m$, that is $m \le \frac{n-1}{2}$. Also, we can write $\det \E = aH+\pi^*L$, for some $a \in \Z, L \in \Pic(B)$. Now
$$\O_M \cong (\det \E)_{|M} \cong ((\det \E)_{|F})_{|M} \cong \O_M(a)$$
and therefore $a=0$. Hence $\det \E = \pi^*L$ and, since $\det \E$ is globally generated and non trivial by Lemma \ref{ulr}(ii), \cite[Lemma 5.1]{lo} shows that there is a vector bundle $\H$ on $B$ such that $\E \cong \pi^*\H$. This proves (i). Next, to see (ii), suppose that $\det \E$ is not $(m-1)$-ample. Then Lemma \ref{q} gives a linear space $M \cong \P^m, M \subset X$ such that $\E_{|M}$ is trivial. Then $\pi_{|M} : M \to B$ is either constant or finite-to-one onto $B$. In the first case, arguing as in the proof of case (i), we reach the same conclusion. In the second case, we get from \cite[Thm.~4.1]{laz} that $B \cong \P^m$ and, if $k=\deg\pi_{|M}$, then $M \cap F$ consists of $k$ points and therefore, being $M$ a linear space, $k \le \deg F$.
\end{proof}

We record the following simple consequence of Lemma \ref{qa} and \cite[Lemma 4.1]{lo}.

\begin{remark}
\label{qa2}
Assume, in the above lemma, that $X$ is a linear $\P^{n-m}$-bundle over $B$. We have:
\begin{itemize}
\item [(i)] If $\det \E$ is not $m$-ample, then $(X, \O_X(1), \E)$ is a linear Ulrich triple over $B$.
\item [(ii)] If $\det \E$ is not $(m-1)$-ample, then either $(X, \O_X(1), \E)$ is a linear Ulrich triple over $B$ or $B \cong \P^m$ and there is there is a linear space $M \cong \P^m, M \subset X$ such that $M$ is a section of $\pi$ and $\E_{|M}$ is trivial.
\end{itemize}
\end{remark}

\section{Ulrich bundles with $c_1(\E)^4=0$}

In this section we will prove Theorem \ref{c1^4=0}.

\renewcommand{\proofname}{Proof of Theorem \ref{c1^4=0}}
\begin{proof}
If $(X,\O_X(1),\E)$ is as in (i) or (ii) of Theorem \ref{c1^4=0}, then $c_1(\E)^4=0$ and $\E$ is Ulrich in the first case by \cite[Prop.~2.1]{es} (or \cite[Thm.~2.3]{be}), in the second case by \cite[Lemma 4.1]{lo}. Moreover $c_2(\E)=0$ in (i) and $c_2(\E)^2=0$ in (ii) since, by definition of linear Ulrich triple, $\E$ is pull-back of a bundle on a variety of dimension at most $3$.

Now assume that $\E$ is Ulrich and $c_1(\E)^4=0$. 

In particular $\E$ is globally generated by Lemma \ref{ulr}(i). If $(X,\O_X(1))=(\P^n,\O_{\P^n}(1))$, we are in case (i) by Lemma \ref{ulr}(iii). Thus, we assume from now on, that $(X,\O_X(1)) \ne (\P^n,\O_{\P^n}(1))$. In particular, we have that $\rho(X) \ge 2$, for otherwise $c_1(\E)^4=0$ implies that $c_1(\E)=0$ and then $(X,\O_X(1))=(\P^n,\O_{\P^n}(1))$ by \cite[Lemma 2.1]{lo}. 

If $n \ge 6$, we can repeat the proof of \cite[Cor.~4]{ls} using now the fact that $c_1(\E)^4=0$ and that $4 \le \lfloor \frac{n}{2} \rfloor +1$. It follows from that proof that $(X,\O_X(1),\E)$ is as in (ii). 

Next, suppose that $n=5$. 

If $c_1(\E)^3=0$, \cite[Cor.~4]{ls} implies that $(X,\O_X(1),\E)$ is as in (ii). 
\renewcommand{\proofname}{Proof}

Therefore, we can henceforth assume that $c_1(\E)^3 \ne 0$ and then $\nu(\det \E)=3$. By \eqref{muk2} we get that $\phi(\E)=2$. Recall that, in the notation of Definition \ref{notgg}, for every $u \in X$, the fiber $F_u \subset X$ of $\Phi$ is a linear subspace of $\P^N$ by \cite[Lemma 2.10]{lms1}. Moreover, $\dim F_u \ge \phi(\E)=2$ for every $u \in X$. We have
\begin{claim}
\label{equi}
If $\dim F_u=2$ for every $u \in X$, the $(X,\O_X(1),\E)$ is as in (ii). 
\end{claim}
\begin{proof}
Follows from \cite[Lemmas 2.10 and 2.12]{lms1}.
\end{proof}
\renewcommand{\proofname}{Proof}
We can therefore assume that
\begin{equation}
\label{nonequi}
\exists u_0 \in X : F_{u_0} \ \hbox{is a linear} \ \P^{k_0}, \ \hbox{for some} \ k_0 \ \hbox{such that} \ 3 \le k_0 \le 4.
\end{equation}
Let $x \in X$ be general point. Since $\dim F_x = \phi(\E)= 2$ we get that
\begin{equation}
\label{lapa2}
\dim F_1(X,x) \ge 1
\end{equation}
hence Lemma \ref{aggiu} applies and we are in one of the cases (a)-(e) of Lemma \ref{aggiu}. Since $\rho(X) \ge 2$, cases (a) and (b.1) are excluded. In case (b.2), \cite[Lemma 5.2]{lo} gives that $(X,\O_X(1),\E)$ is as in (ii). Case (c.1) does not occur, since the classification of Del Pezzo manifolds (see for example \cite[pages 860-861]{lp}, \cite[Table, page 710]{f1}) implies, in the case of dimension $5$, that $\rho(X)=1$. Also, case (c.2) can be excluded: In fact, it follows from \cite[Prop.~2.17]{lms1} that $(X, \O_X(1), \E)=(\P^1 \times Q, \O_{\P^1}(1) \boxtimes \O_Q(1), p^*(\mathcal \G(1)))$, where $p : \P^1 \times Q \to Q=Q_4$ is the second projection and $\G$ is a direct sum of spinor bundles on $Q$. But \cite[Rmk.~2.9]{o} gives that any spinor bundle $\mathcal S$ on $Q$ has $c_1(\mathcal S)=M$ with $M \in |\O_Q(1)|$. Hence $c_1(\G(1))^4 \ne 0$, and therefore also $c_1(\E)^4 \ne 0$, a contradiction. As for case (c.3), let $h : X=\P(\F) \to B$ be the bundle map and consider $f_u=h^{-1}(h(u))$ for any $u \in X$, so that $f_u$ is a linear $\P^3$ contained in $X$. We apply the Dichotomy Lemma \cite[Lemma 2.14]{lms1} to $h$. Note that case (fact) of that lemma does not occur, for otherwise \cite[Lemma 2.16(iii)]{lms1} gives the contradiction $\P^2 \cong F_x=f_x$. Therefore we are in case (fin) of the Dichotomy Lemma, that is $\dim F_u \cap f_u = 0$ for every $u \in X$. Consider $F_{u_0}$ in \eqref{nonequi}. Since $F_{u_0} \not\subset f_{u_0}$, we have that $h_{|F_{u_0}} : F_{u_0} \cong \P^{k_0} \to B$ is not constant, thus it is finite-to one, a contradiction. In case (d.1), we have that $(X,\O_X(1),\E)$ is as in \cite[Cor.~2(ii-2) or (ii-3)]{lms1}. It follows from the proof of \cite[Cor.~2]{lms1} in these cases that we have anyway a linear Ulrich triple over $\P^3$ or $Q_3$, thus again case (ii). Observe now that case (e) does not occur by \cite[Prop.~2.18]{lms1}.

Hence, we are left with cases (d.2)-(d.4). 

In order to study them, we will apply the Dichotomy Lemma \cite[Lemma 2.14]{lms1} to $h=\phi_{\tau}$. We first observe the following.

\begin{claim}
\label{fe2}
In cases (d.2), (d.3) and (d.4), case {\rm (fin)} of the Dichotomy Lemma \cite[Lemma 2.14]{lms1} does not hold for $\phi_{\tau}$.
\end{claim}
\begin{proof}
Suppose that case (fin) holds for $\phi_{\tau}$ and let $L$ be any line such that $x \in L \subseteq F_x$. We know by Lemma \ref{aggiu} than $h(L)=\phi_{\tau}(L)$ is a point, so that $L \subseteq f_x \cap F_x$, a contradiction.
\end{proof}
\renewcommand{\proofname}{Proof}
Therefore the Dichotomy Lemma \cite[Lemma 2.14]{lms1} with $h=\phi_{\tau}$ gives that
\begin{equation}
\label{fact}
F_u \subseteq f_u \ \hbox{for every} \ u \in X
\end{equation}
and we have a commutative diagram
\begin{equation}
\label{comm}
\xymatrix{X \ar[d]_{h} \ar[r]^{\hskip -.3cm \widetilde \Phi} & \widetilde{\Phi(X)} \ar[dl]^{\psi} \\ X'}
\end{equation}
For any $u \in X$, set $T_u= \psi^{-1}(h(u)) \setminus \{\widetilde \Phi(u)\}$. Note that, for every $u \in X$, $\dim T_u = 2$ in case (d.2) and $\dim T_u \ge 1$ in case (d.3). We also get that, set-theoretically,
\begin{equation}
\label{fatto2}
f_u = F_u \sqcup \bigsqcup\limits_{t \in T_u} \widetilde \Phi^{-1}(t).
\end{equation}
Now,
\begin{claim} 
\label{nome}
Cases (d.3) and (d.4) do not occur.
\end{claim}
\begin{proof}
In case (d.3), if $x \in X$ is a general point, we have by \eqref{fact} that $F_x \subseteq f_x$. On the other hand, $F_x$ is a linear plane and $f_x$ is a smooth quadric of dimension $3$, a contradiction. In case (d.4), consider $F_{u_0}$ in \eqref{nonequi}, so that $F_{u_0} \subseteq f_{u_0}$ by \eqref{fact}, a contradiction since $\dim f_{u_0}=2$.
\end{proof}
\renewcommand{\proofname}{Proof}
It remains to exclude case (d.2). 

\begin{claim} 
\label{nomec}
Case (d.2) does not occur.
\end{claim}
\begin{proof} 
Let $f_u$ be a smooth fiber of $h=\phi_{\tau}$, so that $f_u$ is a smooth Del Pezzo $4$-fold. If $\Pic f_u \cong \Z$, it follows from \cite[Lemma 2.16(iii)]{lms1} that $F_u=f_u$, a contradiction since $F_u$ is a linear space. Therefore, the classification of Del Pezzo manifolds (see for example \cite[pages 860-861]{lp}, \cite[Table, page 710]{f1}) implies, in the case of dimension $4$, that $f_u \cong \P^2 \times \P^2$ and $H_{|f_u}=\O_{\P^2}(1) \boxtimes \O_{\P^2}(1)$, that is $f_u$ is embedded in $\P^8$ by the Segre embedding, hence with $\deg f_u = 6$.

Consider \eqref{nonequi} and set $V = f_{u_0}$ and $N=F_{u_0}$. It follows from \cite[(1.5)]{f2} that $V$ is embedded by $H_{|V}$ in $\P^8$ with degree $6$. Hence \cite[(4.6)]{f2} gives that $V$ is a reduced, irreducible normal $4$-fold and it is of type (vu) in \cite[Thm.~(2.9)]{f3}. Moreover \eqref{fact} gives that
\begin{equation}
\label{incl}
N \subset V \subset \P^8
\end{equation}
hence, in particular, since $V$ is irreducible, $N$ is a linear space of dimension $3$. Consider \eqref{fatto2} and set $T = T_{u_0}, P_t = \widetilde \Phi^{-1}(t)$, for any $t \in T$. Then $P_t$ is a linear space of dimension $k_t$ with $k_t \in \{2, 3\}$  (since $V$ is irreducible) and there is a nonempty open subset $U \subseteq T$ such that $k_t=2$ for every $t \in U$. Moreover, set-theoretically, we have
\begin{equation}
\label{fatto3}
V = N \sqcup \bigsqcup\limits_{t \in T} P_t.
\end{equation}
Since $\P^2 \times \P^2$ does not contain any subvariety isomorphic to $\P^3$, it follows from \eqref{incl} that $V$ is singular. The plan is to reach a contradiction by applying to $V$ the arguments in \cite{f3}.

Note that $V$ is not a cone by \cite[page 232]{f2}. Since $V$ is of type (vu) in \cite[Thm.~(2.9)]{f3}, it follows from \cite[(6) and (10)]{f3} that $V \subset W$, where $W$ is a cone, with vertex a plane $R$, over a $2$-Veronese surface $M=W \cap L$, with $L$ a linear space of dimension $5$ disjoint from $R$. We consider the projection $\pi : \P^8 \smallsetminus R \to L$ and the corresponding morphism $\widetilde \pi : \widetilde \P^8 \to L$ obtained by blowing up $R$. Given any subvariety $G \subset \P^8$, we will denote by $\pi(G)$ the image $\widetilde \pi (\widetilde G)$ of the strict transform of $G$. Note that, when $G$ is a linear space, so is $\pi(G)$. Let $t \in U$, so that we know that $P_t$ is a plane. If $R \subseteq P_t$, we have that $R=P_t$ and this can happen for at most one $t \in U$, since the $P_t$'s are disjoint. We get a surface $U'=\{t \in U : R \ne P_t \}$. Let now $t \in U'$. If $R \cap P_t = \emptyset$, then $\pi(P_t)$ is a plane contained in $M$, a contradiction. If $R \cap P_t$ is a point, then $\pi(P_t)$ is a line contained in $M$, again a contradiction. Therefore $R \cap P_t$ is a line for every $t \in U'$. But if $t' \in U', t' \ne t$, the two lines $R \cap P_t, R \cap P_{t'}$ are both contained in the plane $R$, giving rise to the contradiction $\emptyset \ne (R \cap P_t) \cap (R \cap P_{t'}) \subset P_t \cap P_{t'} = \emptyset$.

Therefore case (d.2) is excluded.
\end{proof}
\renewcommand{\proofname}{Proof}

This concludes the proof of the theorem for $n \ge 5$.

For (iii)-(iv), assume $n=4$ and $c_1(\E)^4=0$, so that $c_2(\E)^2=0$ by Lemmas \ref{ulr}(i) and \ref{c3no}(i). If $c_2(\E)=0$ and $r \ge 2$, then we are in case (iii) by \cite[Thm.~2]{blv}. Now suppose that $c_2(\E) \ne 0$. Since $c_1(\E)^4=0$, then $\E$ is not big by \cite[Rmk.~2.2]{lm}. Therefore \cite[Thm.~1]{lms1} applies and we get cases (ii1) and (iii)-(x) of \cite[Thm.~1]{lms1}, hence $(X,\O_X(1),\E)$ is as in (iv). 

This concludes the proof of the theorem.
\end{proof}
\renewcommand{\proofname}{Proof}

\section{Important results when $c_2(\E)^2=0$}

In this section we will prove several technical results that constitute the heart of the paper.

The following lemma will be crucial in several proofs.

\begin{lemma}
\label{mapc}
Let $n \ge 4, d \ge 2$ and let $\E$ be a rank $r$ globally generated bundle on $X \subset \P^N$. Assume that $c_2(\E) \ne 0, c_2(\E)^2=0$ and $\det \E$ is $(n-4)$-ample. Then $c_i(\E)=0$ for $i \ge 3$ and there is a morphism with connected fibers $p : X \to C$ onto a smooth curve $C$. Moreover, we have:
\begin{itemize}
\item[(i)] Let $Z=D_{r-2}(\varphi)$, for a general morphism $\varphi : \O_X^{\oplus (r-1)} \to \E$. Then $Z=\overline Z \cap Y$ is a smooth complete intersection of two divisors on $X$, with $Y \in |\det \E|$ smooth irreducible and $\overline Z = p^* \Delta$ for some effective globally generated divisor $\Delta$ on $C$. 
\item[(ii)] If $\E$ is Ulrich and $Z$ is a general Ulrich subvariety, then 
$$H^i(\O_X(Y-jH-\overline Z))=0 \ \hbox{for} \ j \ge 1, 0 \le i \le n-2.$$ 
Moreover, if $H^1(\O_X(-Y+\overline Z))=0$, then $H^0(\E(-Y+\overline Z)) \ne 0$ and if $H^0(\O_X(-Y+\overline Z))=0$, then $h^0(\E(-Y+\overline Z))=1$. 
\end{itemize}
\end{lemma}
\begin{proof} 
Since $\det \E$ is globally generated and $(n-4)$-ample, using Remark \ref{somm}, we have that $c_1(\E)^4H^{n-4} > 0$ by \cite[(2.1)]{bs}, hence $c_3(\E)=0$ by Lemma \ref{c3no}(ii) and then $c_i(\E)=0$ for $i \ge 3$ by \cite[Lemma 4.1]{blv}. Note that $r \ge 2$, since $c_2(\E) \ne 0$. Using \cite[Lemmas 4.1 and 4.5]{blv}, we find that $Z$ is nonempty and there is an irreducible divisor $Y \in |\det \E|$ containing $Z$. Since $Y=D_{r-1}(\psi)$, where $\psi: W \otimes \O_X \to \E, W \subset H^0(\E)$ a general subspace of dimension $r$, it follows that both $Z$ and $Y$ are smooth. In fact, suppose that $Z$ (respectively $Y$) is singular. Then \cite[Statement (folklore)(i), \S 4.1]{ba} implies that $r \ge 3$ and $D_{r-3}(\varphi)$ (respectively $D_{r-2}(\psi)$) is nonempty of the expected codimension. On the other hand, Porteous' formula (see for example \cite[Thm.~12.4]{eh}), gives that $[D_{r-3}(\varphi)]=c_3(\E)^2-c_2(\E)c_4(\E)=0$ (respectively $[D_{r-2}(\psi)]=c_2(\E)^2-c_1(\E)c_3(\E)=0$), a contradiction. Therefore $Z$ and $Y$ are smooth. 
 
Hence, \cite[Lemma 4.5]{blv} also gives that $\O_Y(Z)$ is globally generated and $Z^2=0$ so that we get, upon passing to the normalization if needed, a morphism $p' : Y \to C'$ onto a smooth curve $C'$ and $Z=(p')^*\Delta'$ for some effective globally generated divisor $\Delta'$ on $C'$. Now, $p'$ extends to a morphism $p_1 : X \to C'$ by \cite[Thm.~5.2.1]{bs} since $Y$ is $(n-4)$-ample. Moreover, passing to the Stein factorization $p : X \to C, g: C \to C'$ of $p_1$, we have that $p$ has connected fibers, and then, if we set $\overline Z=p^*(g^*\Delta')$, we also get that $Z=\overline Z \cap Y$ is a smooth complete intersection of two divisors on $X$. This proves (i). As for (ii), assume now that $\E$ is Ulrich, let $j \ge 1$ and let $i \in \{0, \ldots, n-2\}$. Then the Eagon-Northcott complex gives an exact sequence
$$0 \to \O_X^{\oplus (r-1)}(-j) \to \E(-j) \to \I_{Z/X}(Y-jH) \to 0$$
where we have $H^{i+1}(\O_X^{\oplus (r-1)}(-j))=0$ by Kodaira vanishing and $H^i(\E(-j))=0$ for $i=0$ by definition of Ulrich and for $i \ge 1$ since $\E$ is aCM by Lemma \ref{ulr}(i). It follows that 
\begin{equation}
\label{primovan}
H^i(\I_{Z/X}(Y-jH))=0.
\end{equation}
Now, the exact sequence 
$$0 \to \O_X(-jH-\overline Z) \to \O_X(-jH) \to \O_{\overline Z}(-jH) \to 0$$
shows that $H^i(\O_{\overline Z}(-jH))=0$ since $H^{i+1}(\O_X(-jH-\overline Z))=H^i(\O_X(-jH))=0$ by Kodaira vanishing because $\O_X(\overline Z)$ is globally generated. We deduce from the exact sequence
$$0 \to \O_X(Y-jH-\overline Z) \to \I_{Z/X}(Y-jH) \to \O_{\overline Z}(-jH) \to 0$$
that $H^i(\O_X(Y-jH-\overline Z))=0$. To finish the proof of (ii), assume that $H^1(\O_X(-Y+\overline Z))=0$. From the exact sequence 
$$0 \to \O_X(-Y+\overline Z) \to \I_{Z/X}(\overline Z) \to \O_Y \to 0$$
we deduce that $H^0(\I_{Z/X}(\overline Z)) \ne 0$. Now, the exact sequence
$$0 \to \O_X(-Y+\overline Z)^{\oplus (r-1)} \to \E(-Y+\overline Z) \to \I_{Z/X}(\overline Z) \to 0$$
allows to conclude that $H^0(\E(-Y+\overline Z)) \ne 0$. Also, if $H^0(\O_X(-Y+\overline Z))=0$, then $h^0(\E(-Y+\overline Z))=h^0(\I_{Z/X}(\overline Z))=h^0(\O_Y)=1$, that is (ii).
\end{proof}

We now prove Corollary \ref{c22=0}.

\renewcommand{\proofname}{Proof of Corollary \ref{c22=0}}
\begin{proof}
First note that $d=\deg X \ge 2$ by Lemma \ref{ulr}(iii), since $c_2(\E) \ne 0$. Moreover, $\E$ is globally generated by Lemma \ref{ulr}(i) and $\det \E$ is globally generated and non trivial by Lemma \ref{ulr}(ii), since $c_2(\E) \ne 0$. Suppose first that $\det \E$ is $(n-4)$-ample. Then $\E$ is not $(n-4)$-ample if $r \ge 3$, for otherwise, since $c_3(\E)=0$, if $\varphi: \O_X^{\oplus (r-2)} \to \E$ is a general morphism, then $D_{r-3}(\varphi)=\emptyset$. On the other hand, the same proof of \cite[Thm.~6.4(a)]{t} gives that $D_{r-3}(\varphi) \ne \emptyset$. Moreover, we are in case (ii) by Lemma \ref{mapc}(i) and $\rho(X) \ge 2$. If $\det \E$ is not $(n-4)$-ample, we are in case (i) by Lemma \ref{q}. It follows that $\E$ is not $(n-4)$-ample by Remark \ref{somm} and $\rho(X) \ge 2$, for otherwise $\det \E$ is ample, a contradiction. Next, suppose that $c_3(\E) \ne 0$. To see (i1)-(i2), note that $c_1(\E)^4=0$ by Lemmas \ref{ulr}(i) and \ref{c3no}(ii). Now, if $n \ge 5$, Theorem \ref{c1^4=0} applies with $b=3$ since $c_3(\E) \ne 0$, hence we are in case (i1). If $n=4$, we get by Theorem \ref{c1^4=0} that we are in one of the cases (ii1) and (iii)-(x) of \cite[Thm.~1]{lms1}. Moreover the cases (iii) and (v2) of \cite[Thm.~1]{lms1} are excluded since $c_3(\E) \ne 0$. This gives (i2).
\end{proof}
\renewcommand{\proofname}{Proof}

We now study the case on $\P^k$- bundles.

\begin{lemma}
\label{pfgen}
Let $B$ be a smooth irreducible variety of dimension $b \ge 1$ and let $\F$ be a rank $n-b+1$ very ample bundle on $B$ with $n \ge \max\{4, b+2\}$. Let $X=\P(\F) \subset \P H^0(H)$, where $H=\xi$ is the tautological line bundle and let $\pi : X \to B$ be the bundle map. Let $\E$ be a rank $r$ Ulrich bundle with $c_2(\E) \ne 0, c_2(\E)^2=0$ and let $\det \E=a \xi + \pi^*M_1$ for some $M_1 \in \Pic(B)$. If $\det \E$ is $(n-4)$-ample, then $c_3(\E)=0$ and we have:
\begin{itemize}
\item[(i)] $0 \le a \le b$.
\item[(ii)] If $a=0$, $(X, \O_X(1),\E)$ is linear Ulrich triple over $B$.
\item[(iii)] If $a=b$, then $b=1$ and there are a line bundle $M$ and a rank $r-1$ bundle $\G$ on $B$ with $H^i(M)=H^i(\G)=0$ for $i \ge 0$, such that, if $\L=\xi+\pi^*M$, then $\E$ sits in an exact sequence of Ulrich bundles
$$0 \to \L \to \E \to \pi^*(\G(\det \F)) \to 0.$$  
\end{itemize}
Vice versa, any vector bundle $\E$ in (iii) is a rank $r$ Ulrich bundle with $c_2(\E) \ne 0$ and $c_2(\E)^2=0$. 
\end{lemma}
\begin{proof}
Note that $d= \deg X \ge 2$, for otherwise $c_2(\E)=0$ by Lemma \ref{ulr}(iii). Moreover, $\E$ and $\det \E$ are globally generated by Lemma \ref{ulr}(i)-(ii), hence $a \ge 0$. If $a=0$, then $\det \E =\pi^*M_1$ so that \cite[Lemmas 5.1 and 4.1]{lo} imply that $(X, \O_X(1),\E)$ is linear Ulrich triple over $B$. This proves (ii) and the first inequality in (i). As for the rest, we will apply Lemma \ref{mapc}. We get that $c_3(\E)=0$ and there is a morphism with connected fibers $p : X \to C$ onto a smooth curve $C$. If $F \cong \P^{n-b}$ is any fiber of $\pi$, then $p_{|F} : F \to C$ must be constant because $n-b>1$, hence we can apply \cite[Lemma 1.15(b)]{de} and get a morphism $\psi : B \to C$ such that $p=\psi \circ \pi$. Also, it follows from Lemma \ref{mapc} that, if $Z$ is a general Ulrich subvariety associated to $\E$, there are divisors $Y \in |\det \E|$ and $\Delta$ on $C$ such that, if $\overline Z=p^* \Delta$, then 
$$H^i(\O_X(Y-jH-\overline Z))=0 \ \hbox{for} \ j \ge 1, 0 \le i \le n-2.$$ 
Hence, setting $M_2=\psi^*(\O_C(\Delta))$ and $M=M_1-M_2$, we find 
\begin{equation}
\label{laj-gen}
H^i(\O_{\P(\F)}((a-j)\xi+\pi^*M))=0 \ \hbox{for} \ j \ge 1, 0 \le i \le n-2.
\end{equation}
For any $j \le n+a-b$, we have that $a-j \ge b-n$, hence $R^q \pi_*(\O_{\P(\F)}((a-j)\xi))=0$ for $q >0$, so that 
\begin{equation}
\label{laj-gen3}
H^i(\O_{\P(\F)}((a-j)\xi+\pi^*M))=H^i(\pi_*(\O_{\P(\F)}((a-j)\xi))(M)) \ \hbox{for} \ i \ge 0 \ \hbox{and} \ 1 \le j \le n+a-b.
\end{equation}
Hence, observing that the latter vanish for $i \ge n-1$, combining with \eqref{laj-gen} we have that 
\begin{equation}
\label{laj-gen2}
H^i(\O_{\P(\F)}((a-j)\xi+\pi^*M))=0 \ \hbox{for} \ i \ge 0 \ \hbox{and} \ 1 \le j \le n+a-b.
\end{equation}
This implies that $\chi(\O_{\P(\F)}((a-j)\xi+\pi^*M))=0$ for $1 \le j \le n+a-b$. Now assume that $b<a$. Since $\chi(\O_{\P(\F)}((a-m)\xi+\pi^*M))$ is a non-zero polynomial in $m$, it has at most $n$ roots, while $n+a-b>n$, a contradiction. This proves (i). If $a=b$, \eqref{laj-gen2} shows that $\L$ is Ulrich and \eqref{laj-gen3} together with \eqref{laj-gen} give that $H^i(M \otimes S^{b-j} \F)=0$ for $i \ge 0$ and $1 \le j \le b$. Moreover, observe that $\L=\O_X(Y-\overline Z)$ and $H^1(-\L)=H^1(-b\xi - \pi^*M)=0$ because $R^q \pi_*(-b\xi - \pi^*M)=0$ for $q \le 1$. Hence Lemma \ref{mapc}(ii) gives that $H^0(\E(-\L)) \ne 0$ and we get an exact sequence
$$0 \to \L \to \E \to \mathcal Q \to 0$$
where $\mathcal Q$ is the quotient sheaf. Since $\L$ and $\E$ are Ulrich vector bundles for $(X, \xi)$, then so is $\mathcal Q$ by \cite[Prop.~3.3.1]{cmrpl}. Moreover, $\det \mathcal Q = \pi^* M_2$ and then it follows from \cite[Lemmas 5.1 and 4.1]{lo} that $\mathcal Q \cong \pi^*(\G(\det \F))$, where $\G$ is a rank $r-1$ vector bundle on $B$ such that $H^i(\G \otimes S^k \F^*)=0$ for all $i \ge 0, 0 \le k \le b-1$. Also, $c_1(\G(\det \F))=M_2$ and $c_1(\G(\det \F))^2=M_2^2=0$. If $b \ge 2$, \cite[Cor. 5]{ls} gives a contradiction. This proves (iii). The last assertion follows from \cite[Ex.~10.1]{blv}.
\end{proof}

When $(X, \O_X(1))$ is a linear $\P^{n-1}$-bundle over a smooth curve and $n \ge 4$, we will see that the Ulrich bundles with $c_2(\E) \ne 0, c_2(\E)^2=0$ can be precisely classified.

\begin{lemma}
\label{pf}
Let $B$ be a smooth irreducible curve of genus $g$ and let $\F$ be a rank $n \ge 4$ very ample bundle on $B$. Let $X=\P(\F) \subset \P H^0(H)$, where $H=\xi$ is the tautological line bundle and let $\pi : X \to B$ be the bundle map. Let $\E$ be a rank $r$ Ulrich bundle with $c_2(\E) \ne 0$ and $c_2(\E)^2=0$. Then we have:
\begin{itemize}
\item[(i)] If $n \ge 5$ or if $n=4$ and $\det \E$ is ample, there are a line bundle $M \in \Pic^{g-1}(B)$ with $H^i(M)=0$ for $i \ge 0$ and a rank $r-1$ bundle $\G$ on $B$ with $H^i(\G)=0$ for $i \ge 0$ so that, if $\L=\xi+\pi^*M$, then $\E$ sits in an exact sequence
$$0 \to \L \to \E \to \pi^*(\G(\det \F)) \to 0.$$ 
Moreover, $c_3(\E)=0$ and $(X, \O_X(1), \E)$ is as in \cite[Ex.~10.1]{blv}. 
\item[(ii)] If $n=4$ and $\det \E$ is not ample, then $(X, \O_X(1), \E) \cong (\P^1 \times \P^3, \O_{\P^1}(1) \boxtimes  \O_{\P^3}(1), \pi_2^*(\O_{\P^3}(1))^{\oplus r})$, where $\pi_2 : X \cong \P^1 \times \P^3 \to \P^3$ is the second projection and $c_3(\E) \ne 0$.
\end{itemize}
Vice versa, any vector bundle $\E$ in (i) or (ii) is a rank $r$ Ulrich bundle with $c_2(\E) \ne 0, c_2(\E)^2=0$. 
\end{lemma}
\begin{proof}
Note that $c_2(\E) \ne 0$ implies that $(X,\O_X(1),\E)$ is not a linear Ulrich triple over $B$ and, by Lemma \ref{ulr}(iii), that $d=\xi^n \ge 2$. To see (i), assume that either $n \ge 5$ or that $n=4$ and $\det \E$ is ample. Observe that $\det \E$ is $(n-4)$-ample, for otherwise $n \ge 5$ and $\det \E$ is not $1$-ample, but then Remark \ref{qa2}(i) would give that $(X,\O_X(1),\E)$ is a linear Ulrich triple over $B$. Hence, we can apply Lemma \ref{pfgen} and we get that $c_3(\E)=0$ and that (i) holds. Now, suppose that $n=4$ and $\det \E$ is not ample. It follows by Remark \ref{qa2}(ii) that $B \cong \P^1$ and there is a line $L \subset X$ such that $\pi_{|L} : L \to \P^1$ is an isomorphism and $\E_{|L}$ is trivial. Letting $f$ be the class of a fiber we have that $d=\xi^n=\deg \F, f^2=0, \xi^{n-1}f=1$ and $K_X=-n\xi+(d-2)f$. Writing $\det \E=a\xi+m_1f$, for some $m_1 \in \Z$, and using $\xi \cdot L = f \cdot L =1$ and $(\det \E) \cdot L = 0$, we get that $m_1=-a$. Now, Lemma \ref{ulr}(v) gives that
$$a(d-1)=a(\xi-f)\xi^{n-1}=c_1(\E)\xi^{n-1}=\frac{r}{2}(K_X+(n+1)\xi)\xi^{n-1}=\frac{r}{2}(\xi+(d-2)f)\xi^{n-1}=r(d-1)$$
and therefore $a=r$ and $\det \E = r (\xi-f)$. Note that $\E^*(K_X+(n+1)\xi)$ is Ulrich by Lemma \ref{ulr}(viii) and
$$\det (\E^*(K_X+(n+1)\xi))=r(d-1)f=\pi^*(\O_{\P^1}(r(d-1)).$$
Since $\det (\E^*(K_X+(n+1)\xi))$ is globally generated and non trivial by Lemma \ref{ulr}(ii), it follows from \cite[Lemmas 5.1 and 4.1]{lo} that 
\begin{equation}
\label{ud}
\E^*(K_X+(n+1)\xi) \cong \pi^*(\G(\det \F))
\end{equation}
where $\G$ is a rank $r$ vector bundle on $\P^1$ such that $H^i(\G)=0$ for every $i \ge0$. Therefore $\G \cong \O_{\P^1}(-1)^{\oplus r}$ and we deduce from \eqref{ud} that 
$$\E \cong (\xi-f)^{\oplus r}.$$
Hence 
$$0 = c_2(\E)^2 = \binom{r}{2}^2(\xi-f)^4=\binom{r}{2}^2(d-4)$$
so that $d=4$. Since $\F \cong \bigoplus_{i=1}^4 \O_{\P^1}(a_i)$ with $a_i \ge 1$ for every $i$, we get that $a_i=1$ for every $i$ and then $(X, \O_X(1)) \cong (\P^1 \times \P^3, \O_{\P^1}(1) \boxtimes \O_{\P^3}(1))$. Now $\xi-f = \pi_2^*(\O_{\P^3}(1))$ and we conclude that $\E \cong \pi_2^*(\O_{\P^3}(1))^{\oplus r}$. Also, $c_3(\E)=\binom{r}{3}[\P^1 \times \{z\}] \ne 0$.
 
The last assertion follows from \cite[Ex.~10.1]{blv} and \cite[(3.5)]{be}.
\end{proof}

When $n \ge 4$ and $(X, \O_X(1))$ is a quadric fibration over a smooth curve $B$, we will see that there are no Ulrich bundles with $c_2(\E) \ne 0, c_2(\E)^2=0$, unless $n=4$ and $B \cong \P^1$.

\begin{lemma}
\label{qf}
Let $n \ge 4$ and let $(X, \O_X(1))$ be a hyperquadric fibration $\pi : X \to B$ over a smooth curve $B$. Let $\E$ be an Ulrich bundle with $c_2(\E) \ne 0, c_2(\E)^2=0$. Then $n=4, \det \E$ is not ample and $B \cong \P^1$. 
\end{lemma}
\begin{proof}
First, observe that if $n \ge 5$, then $\det \E$ is $(n-4)$-ample. In fact, if not, $\det \E$ is not $1$-ample and Lemma \ref{qa}(i) implies that $\E \cong \pi^*\H$ and therefore $c_2(\E)=0$, a contradiction.

Now assume that either $n \ge 5$ or $n=4$ and $\det \E$ is ample. Then $\det \E$ is $(n-4)$-ample, hence Lemma \ref{mapc} gives a morphism with connected fibers $p : X \to C$ onto a smooth curve $C$. Let $F$ be any fiber of $\pi$, so that $\Pic F \cong \Z \O_F(1)$ (see for example \cite[Cor.~2.3.4]{bs}). Then $p_{|F} : F \to C$ must be constant, otherwise there would be a non trivial line bundle $N$ on $F$ with $N^2=0$, a contradiction. Hence $p$ contracts all fibers of $\pi$ to a point, so that \cite[Lemma 1.15(b)]{de} gives a morphism $\psi : B \to C$ such that $p=\psi \circ \pi$. Since $p$ has connected fibers, it follows that $\psi$ is an isomorphism and we can assume that $\pi=p$.

Using Lemma \ref{qa}, we have that 
$$\det \E=cH+\pi^*M_1$$
for some $M_1 \in \Pic(B)$ and set $m_1= \deg M_1$. Since $\det \E$ is globally generated and non trivial by Lemma \ref{ulr}(ii), we have that $c \ge 1$: Indeed, if $L$ is any line in a fiber $F$, then $0 \le \det \E \cdot L = c$ and if $c=0$ we have that $c_1(\E)^2=0$, giving, by \cite[Lemma 4.2(i)]{blv}, the contradiction $c_2(\E)=0$.

Recall now from \cite[\S 2]{la} that $\F=\pi_*\O_X(H)$ is a rank $n+1$ vector bundle on $B$ such that $X$ can be embedded fiberwise in $\P(\F)$ as a divisor of numerical class $2\xi+(d-2e)\widetilde F$, where $\xi$ is the tautological line bundle of $\P(\F)$, $\widetilde F$ is a fiber of the morphism $\P(\F) \to B, \xi_{|X}=H, e = \deg \F$ and $d=H^n$. Moreover, if $g$ is the genus of $B$, we have as in \cite[\S 2]{la}, that
$$F^2=0, \ H^{n-1}F=2 \ \hbox{and} \ K_X \equiv -(n-1)H+(d+2g-2-e)F.$$ 
Furthermore, $\F$ is globally generated by \cite[\S 3]{lmp}, hence, in particular $e \ge 0$. 

We now claim that $c=1$. To this end, Lemma \ref{mapc} gives that, if $Z$ is a general Ulrich subvariety associated to $\E$, then $Z=\overline Z \cap Y$, for two divisors $Y \in |\det \E|$ and $\overline Z \in |\pi^*M_2|$, for some $M_2 \in \Pic(B)$. Also, we have by Lemma \ref{mapc}(ii) that
$$H^i(\O_X(Y-jH-\overline Z))=0 \ \hbox{for} \ 0 \le i \le 1 \ \hbox{and} \ j \ge 1$$
that is, setting $M=M_1-M_2$,
$$H^i(\O_X((c-j)H+\pi^*M))=0 \ \hbox{for} \ 0 \le i \le 1 \ \hbox{and} \ j \ge 1.$$
Since $R^1 \pi_*(\O_X((c-j)H+\pi^*M))=0$ because $H^1(\O_Q(c-j))=0$ on any fiber $Q$, we find that, 
\begin{equation}
\label{laj2}
H^i(\pi_*(\O_X((c-j)H))(M))=0 \ \hbox{for} \ 0 \le i \le 1 \ \hbox{and} \ j \ge 1.
\end{equation}
Setting $j=c$ in \eqref{laj2}, we get that
\begin{equation}
\label{rrc2}
H^i(M)=0 \ \hbox{for} \ i \ge 0
\end{equation}
and in particular, setting $m_2= \deg M_2$,
\begin{equation}
\label{dm}
m_1-m_2=\deg M = g-1.
\end{equation} 
If $c \ge 2$, setting $j=c-1$ in \eqref{laj2}, we have that $H^i(\F(M))=0$ for $i \ge 0$ and therefore $\chi(\F(M))=0$, that is, using \eqref{dm}, $e+(n+1)(g-1)=(n+1)(g-1)$, so that $e=0$. Then, being $\F$ globally generated, we have that $\F \cong \O_B^{\oplus (n+1)}$, $\P(\F) \cong B \times \P^n$ with $\pi=\pi_1$ the first projection and $\xi=\pi_2^*\O_{\P^n}(1)$, where $\pi_2$ is the second projection. Now let $y \in \P^n$ and consider the section $B \times \{y\} \subset B \times \P^n$. We claim that $B \times \{y\}$ can meet $X$ in at most one point. In fact, consider the restriction map $H^0(\xi) \to H^0(H)$. Since $X \sim  2\xi+\widetilde \pi^*N$ for some $N \in \Pic(B)$, we have that $H^0(\xi-X)=H^0(-\xi-\widetilde \pi^*N)=0$. Moreover, $h^0(\xi)=h^0(\F)=h^0(\pi_*H)=h^0(H)$ and therefore the restriction map $H^0(\xi) \to H^0(H)$ is an isomorphism. It follows that the restriction to $X$ of the morphism $\varphi_{\xi} : \P(\F) \to \P H^0(\xi)$ is the embedding $\varphi_H : X \to \P H^0(H)$. Since $\varphi_{\xi}$ contracts $B \times \{y\}$ to a point, it follows that $B \times \{y\}$ can meet $X$ in at most one point. Therefore 
$$d=(2\xi+d\widetilde F) \cdot (B \times \{y\})=X \cdot (B \times \{y\}) \le 1$$
a contradiction since then $d=1$ and therefore $c_2(\E)=0$ by Lemma \ref{ulr}(iii). This proves that $c=1$. 

Now, Lemma \ref{ulr}(v) gives
$$\begin{aligned}[t] 
d+2m_1 & = (H+m_1F)H^{n-1}=c_1(\E)H^{n-1}=\frac{r}{2}(K_X+(n+1)H)H^{n-1} =
\\ & =\frac{r}{2}(2H+(d+2g-2-e)F))H^{n-1}=r(2d+2g-2-e)
\end{aligned}$$
and therefore
\begin{equation}
\label{11}
m_1=\frac{1}{2}(r(2d+2g-2-e)-d).
\end{equation}
We now compute $c_2(\E)H^{n-2}$ using Lemma \ref{ulr}(vi). First
\begin{equation}
\label{c21}
c_2(\E)=[Z]=[Y][\overline Z]=(H+m_1F)m_2F=m_2 HF.
\end{equation}
Using the standard exact sequences for $\Omega_{\P(\F)/B}$ and $\Omega_{\P(\F)}$, we find that 
$$c_1(\E)^2H^{n-2}=d+4m_1, \ c_1(\E)K_XH^{n-2}=-(d+2m_1)(n-1)+2(d+2g-2-e)$$
$$K_X^2H^{n-2}=(n-1)^2d-4(d+2g-2-e)(n-1)$$
$$c_2(X)=\frac{n^2-3n+4}{2}H^2+(-2+3d-4e+2g+2n-dn+en-2gn) HF$$
and then
$$c_2(X)H^{n-2}=\frac{d(n^2-3n+4)}{2}-4+6d-8e+4g+4n-2dn+2en-4gn.$$
Using the above, \eqref{c21} and Lemma \ref{ulr}(vi), we find
$$2m_2=m_2 H^{n-1}F=c_2(\E)H^{n-2}=\frac{1}{2}(4-3d+2e-4g-4r+4dr-3er+4gr)$$
so that
$$m_2=\frac{1}{4}(4-3d+2e-4g-4r+4dr-3er+4gr).$$
Hence, using the above, \eqref{dm} and \eqref{11}, we see that
$$\frac{1}{2}(r(2d+2g-2-e)-d)-g+1=m_1-g+1=m_2=\frac{1}{4}(4-3d+2e-4g-4r+4dr-3er+4gr)$$
that is $d+e(r-2)=0$, a contradiction since $r \ge 2$ because $c_2(\E) \ne 0$.

We therefore deduce that $n=4$ and $\det \E$ is not ample. Now, observing that, as above, we cannot have that $\E \cong \pi^*\G$, Lemma \ref{qa}(ii) implies that $B \cong \P^1$. 
\end{proof}

In the case $n=4, \det \E$ is not ample and $B \cong \P^1$, we do have an example, that is the unique one if $c_1(\E)^4=0$ by \cite[Prop.~2.17]{lms1}. We don't know if there are examples with $c_1(\E)^4>0$.

\begin{example}
If $(X, \O_X(1)) \cong (\P^1 \times Q_3, \O_{\P^1}(1) \boxtimes  \O_{Q_3}(1))$, then the first projection $\pi_1 : X \to \P^1$ gives a hyperquadric fibration. Let $\H$ be a direct sum of spinor bundles on $Q_3$ and let $\E= \pi_2^*(\H(1))$, where $\pi_2 : X \to Q_3$ is the second projection. Then $\E$ is an Ulrich bundle on $X$ with $c_1(\E)^4=0$ and $c_3(\E) \ne 0$.
\end{example}

When $(X, \O_X(1))$ is a linear $\P^3$-bundle over the plane, we will see that the Ulrich bundles with $c_2(\E) \ne 0, c_2(\E)^2=0$ and $\det \E$ is not $1$-ample can be precisely classified.

\begin{lemma}
\label{pf2}
Let $\F$ be a rank $4$ very ample bundle on a smooth surface $B$. Let $X=\P(\F) \subset \P H^0(H)$, where $H=\xi$ is the tautological line bundle and let $\pi : X \to B$ be the bundle map. Let $\E$ be a rank $r$ Ulrich bundle with $c_2(\E) \ne 0, c_2(\E)^2=0$ and assume that $\det \E$ is not $1$-ample. Then either $(X, \O_X(1), \E)$ is a linear Ulrich triple over $B$ or $B \cong \P^2$, $(X, \O_X(1), \E) \cong (\P^2 \times \P^3, \O_{\P^2}(1) \boxtimes \O_{\P^3}(1), \pi_2^*(\O_{\P^3}(2))^{\oplus r})$, where $\pi_2 : X \cong \P^2 \times \P^3 \to \P^3$ is the second projection and $c_3(\E) \ne 0$.
\end{lemma}
\begin{proof}
Assume that $(X, \O_X(1), \E)$ is not a linear Ulrich triple over $B$. Since $\det \E$ is not $1$-ample, it follows from Remark \ref{qa2}(ii) that $B \cong \P^2$ and there is a linear space $P \cong \P^2, P \subset X$ such that $P$ is a section of $\pi$ and $\E_{|P}$ is trivial. Set $c_i(\F)=c_i H_{\P^2}^i, 1 \le i \le 2$ and $R = \pi^* \O_{\P^2}(1)$, so that, if $f$ is the numerical class of a fiber of $\pi$, we have
$$d=\xi^5=c_1^2-c_2, \ R^3=0, \ R^2=f, \ \xi^3f=1, \ \xi^4 R=c_1 \ \hbox{and} \ K_X=-4\xi+(c_1-3)R.$$ 

Write $\det \E = a\xi+bR$. If $L \subset P$ is a line, then, being $(\det \E)_{|P} \cong \O_P$, we see that $0 = \det \E \cdot L = a+b$, and therefore $b=-a$. Since $\det \E$ is non trivial by Lemma \ref{ulr}(ii), we have
\begin{equation}
\label{dete2}
\det \E = a(\xi-R), \ \hbox{with} \ a \ge 1.
\end{equation}
Now consider a line $L \subset \P^2$, so that $\F_{|L} \cong \bigoplus_{i=0}^4 \O_{\P^1}(a_i)$ with $a_1 \ge 1$ for all $i$ since $\F_{|L}$ is ample and $a_1+a_2+a_3+a_4=c_1(\F_{|L})=c_1(\F) \cdot L=c_1$. Therefore $c_1 \ge 4$ and equality holds if and only if $a_i=1$ for all $i$. From \eqref{dete2} we find that
\begin{equation}
\label{7}
c_1(\E)^4 \xi = a^4(\xi-R)^4 \xi =\frac{a^2(c_1-4)(3c_1-8)}{4}.
\end{equation}
If $c_1=4$, then $\F_{|L} \cong \bigoplus_{i=0}^4 \O_{\P^1}(1)$ for any line $L$. It follows from \cite[Thm.~I.3.2.1]{oss} that $\F \cong \O_{\P^2}(1)^{\oplus 4}$ and therefore $c_2=6$ and $(X, \O_X(1)) \cong (\P^2 \times \P^3, \O_{\P^2}(1) \boxtimes \O_{\P^3}(1))$. Moreover \eqref{7} gives that $c_1(\E)^4 \xi = 0$ and then $c_1(\E)^4=0$, hence $\det \E$ is not big. Since $c_1(\E)^3 \xi^2 = a^3(\xi-R)^3 \xi^2 =a^3 > 0$, it follows from \cite[Cor. 4.9]{ls} that $\E \cong \pi_2^*(\O_{\P^3}(2))^{\oplus r}$. Also, $c_3(\E)=8\binom{r}{3}[\P^2 \times \{z\}] \ne 0$.

Next, we will show that we reach a contradiction if $c_1 \ge 5$. 

Write $c_2(\E) = \alpha \xi^2+\beta \xi R + \gamma f$, for some $\alpha, \beta, \gamma \in \Z$. Since $c_2(\E)^2=0$
we get
\begin{equation}
\label{1}
\alpha^2 \xi^4+\beta^2 \xi^2 f+2\alpha \beta \xi^3 R + 2\alpha \gamma \xi^2f=0.
\end{equation}
Intersecting in \eqref{1} with $R$, we deduce that
\begin{equation}
\label{2}
\alpha(\alpha c_1+2 \beta)=0
\end{equation}
while, intersecting in \eqref{1} with $\xi$, we find, using \eqref{2}, that
\begin{equation}
\label{3}
\beta^2+2\alpha \gamma-c_2\alpha^2=0.
\end{equation}
From \eqref{2} we have that either $\alpha=0$ or $\alpha \ne 0$ and $\beta = -\frac{\alpha c_1}{2}$ and therefore, using \eqref{3}, either 
$$c_2(\E) = \gamma f$$ 
or 
$$c_2(\E) = \frac{\alpha}{8}(8\xi^2-4c_1 \xi R+(4c_2-c_1^2)f).$$
On the other hand, from $c_2(\E_{|P})=0$, we get that $c_2(\E)[P]=0$. Since $[P]f=1$, it cannot be that $c_2(\E) = \gamma f$, for otherwise $0=c_2(\E)[P] = \gamma$, and therefore $c_2(\E)=0$, a contradiction. Hence $\alpha \ne 0$, $c_2(\E) = \frac{\alpha}{8}(8\xi^2-4c_1 \xi R+(4c_2-c_1^2)f)$ and, using $c_2(\E)[P]=0, \xi^2 P=1, \xi R P =1$, we get that
$$\frac{\alpha}{8}(8-4c_1+4c_2-c_1^2)=0$$ 
that is
\begin{equation}
\label{4}
c_2=\frac{c_1^2}{4}+c_1-2.
\end{equation}
and therefore
\begin{equation}
\label{5}
c_2(\E) = \frac{\alpha}{2}(2\xi^2-c_1 \xi R+(c_1-2)f).
\end{equation}
Using \eqref{4}, we get
$$c_1(\E)\xi^4=a(\xi-R)\xi^4=\frac{a(3c_1^2-8c_1+8)}{4}$$
and
$$\frac{r}{2}(K_X+6\xi)\xi^4=\frac{r}{2}(2\xi+(c_1-3)R))\xi^4=\frac{r(5c_1^2-10c_1+8)}{4}$$
and therefore
Lemma \ref{ulr}(v) gives that
\begin{equation}
\label{6}
a=\frac{r(5c_1^2-10c_1+8)}{3c_1^2-8c_1+8}.
\end{equation}
We now compute $c_2(\E)\xi^3$ using Lemma \ref{ulr}(vi). Using the exact sequences for $\Omega_{X/\P^2}$ and $\Omega_X$, we get
$$c_2(X)=6\xi^2+(12-3c_1)\xi R+\frac{c_1^2-8c_1+4}{4}f.$$
Moreover, we have
$$c_1(\E)^2 \xi^3=\frac{3a^2(c_1-2)^2}{4}, \ c_1(\E)K_X \xi^3=a(-2c_1^2+4c_1-5)$$
$$K_X^2 \xi^3=5c_1^2+2c_1+41, \ c_2(X) \xi^3=\frac{7c_1^2+16c_1+52}{4}.$$
Using the above, \eqref{5} and Lemma \ref{ulr}(vi), we find
$$\frac{\alpha(c_1^2-2c_1+4)}{4}=c_2(\E)\xi^3=\frac{1}{8}(20a+12a^2-16ac_1-12a^2c_1+8ac_1^2+3a^2c_1^2- 
32r+38c_1r-21c_1^2r)$$
and therefore
$$\alpha=\frac{20a+12a^2-16ac_1-12a^2c_1+8ac_1^2+3a^2c_1^2-32r+38c_1r-21c_1^2r}{2(c_1^2-2c_1+4)}.$$
It follows from \eqref{7} that $c_1(\E)^4 \ne 0$, hence Lemma \ref{c3no} gives that $c_3(\E)=0$. We have
$$c_1(\E)c_2(\E)\xi^2=\frac{a \alpha(c_1-2)^2}{4}, \ c_1(\E)^3\xi^2=\frac{a^3(c_1-2) (3c_1-10)}{4}$$
$$c_1(\E)^2K_X\xi^2=-a^2(c_1-2)(2c_1-3), \ c_2(\E)K_X\xi^2=-\frac{\alpha(c_1^2-c_1+8)}{2}$$ 
$$c_1(\E)K_X^2\xi^2=a(5c_1^2-6c_1+17), \ c_1(\E)c_2(X)\xi^2=\frac{a(7c_1^2+4c_1+4)}{4},$$
$$K_Xc_2(X)\xi^2=-4c_1^2-13c_1-88.$$
Applying Lemma \ref{ulr}(vii) we get, dividing by $r(c_1-2)$,
$$\begin{aligned}[t] 
&32768-79872c_1+32256c_1^2+141824c_1^3-289536c_1^4+277920c_1^5-159400c_1^6+56064c_1^7 
\\ & -11214c_1^8+972c_1^9-36864r+73728c_1 r+37632c_1^2 r-317184c_1^3 r+524064c_1^4 r-472800c_1^5 r \\
& +263112c_1^6 r-90516c_1^7 r+17640c_1^8 r-1485c_1^9 r+4096r^2+7168c_1 r^2-68864c_1^2 r^2+164416c_1^3 r^2 \\
& -212480c_1^4 r^2+171280c_1^5 r^2-88400c_1^6 r^2+28300c_1^7 r^2-5000c_1^8 r^2+ 375c_1^9 r^2=0.
\end{aligned}$$
But this equation has no integer solutions when $r \ge 2, c_1 \ge 5$ (see \cite[Out(13)]{m}). This proves the lemma.
\end{proof}

We end the section by studying the case of a $\P^{n-2}$-bundle over a surface.

\begin{lemma}
\label{sup}
Let $B$ be a smooth irreducible surface and let $\F$ be a rank $n-1 \ge 3$ very ample bundle on $B$. Let $X=\P(\F) \subset \P H^0(H)$, where $H=\xi$ is the tautological line bundle and let $\pi : X \to B$ be the bundle map. Let $\E$ be a rank $r$ Ulrich bundle with $c_2(\E) \ne 0, c_2(\E)^2=0$. Then one of the following occurs:
\begin{itemize}
\item[(i)] $(X, \O_X(1), \E)$ is a linear Ulrich triple over $B$.
 \item[(ii)] $B \cong \P^2$, $(X, \O_X(1), \E) \cong (\P^2 \times \P^3, \O_{\P^2}(1) \boxtimes \O_{\P^3}(1), \pi_2^*(\O_{\P^3}(2))^{\oplus r})$, where $\pi_2 : X \cong \P^2 \times \P^3 \to \P^3$ is the second projection and $c_3(\E) \ne 0$. 
\item[(iii)] $n=4$ and $\det \E$ is not ample. 
\end{itemize}
\end{lemma}
\begin{proof}
Arguing by contradiction, assume that we are not in any of the three cases (i)-(iii). 

We can write, for some $M_1 \in \Pic(B)$,  
\begin{equation}
\label{dete22}
\det \E = a \xi + \pi^*M_1.
\end{equation}
Note that $a \ge 1$ because $\det \E$ is globally generated by Lemma \ref{ulr}(ii), hence $a \ge 0$. If $a=0$ then $\det \E =\pi^*M_1$ and \cite[Lemmas 5.1 and 4.1]{lo} imply that $(X, \O_X(1),\E)$ is linear Ulrich triple over $B$, that we have excluded.

Now, observe that $\det \E$ is $(n-4)$-ample. In fact, this is just the assumption if $n=4$ while, if $n=5$, it follows from Lemma \ref{pf2}. When $n \ge 6$, if $\det \E$ is not $(n-4)$-ample, then it is not $2$-ample and Lemma \ref{qa}(i) implies that $\E \cong \pi^*\H$ and therefore \cite[Lemmas 5.1 and 4.1]{lo} give that $(X, \O_X(1), \E)$ is a linear Ulrich triple over $B$, that we have excluded.

Hence $\det \E$ is $(n-4)$-ample and Lemma \ref{pfgen} applies, so that $c_3(\E)=0$ and $a=1$. Since $d \ge 2$, for otherwise $c_2(\E)=0$ by Lemma \ref{ulr}(iii), Lemma \ref{mapc} applies and there is a morphism with connected fibers $p : X \to C$ onto a smooth curve $C$. If $F \cong \P^{n-2}$ is any fiber of $\pi$, $p_{|F} : F \to C$ must be constant, hence $p$ contracts all fibers of $\pi$ to a point and \cite[Lemma 1.15(b)]{de} gives a morphism $\psi : B \to C$ such that $p=\psi \circ \pi$. Also, it follows from Lemma \ref{mapc} that, if $Z$ is a general Ulrich subvariety associated to $\E$, there is a smooth irreducible divisor $Y \in |\det \E|$ and a globally generated divisor $\Delta$ on $C$ such that if $\overline Z=p^*\Delta=\pi^*(\psi^*\Delta)$, then $Z=\overline Z \cap Y$. Furthermore, Lemma \ref{mapc}(ii) gives that
$$H^i(\O_X(Y-H-\overline Z))=0 \ \hbox{for} \ 0 \le i \le 2.$$
Since $Y \sim \xi + \pi^*M_1$ and $\overline Z \sim \pi^*M_2$, with $M_2 = \psi^*\O_C(\Delta)$, we have that $M_2^2=0$ and, setting $M=M_1-M_2$, we see that $H^i(\O_{\P(\F)}(\pi^*M))=0$ for $0 \le i \le 2$. Hence $H^i(M)=0$ for $i \ge 0$ and, in particular,
\begin{equation}
\label{laj22}
\chi(M)=0.
\end{equation}
We now compute the various identities as in Lemma \ref{ulr}(v)-(vii). We will perform the calculations in Mathematica. In the Mathematica code \cite{m1}, we have set the following notation:
$$m12=M_1^2, \ m1m2=M_1M_2, \ c12= c_1(\F)^2, \ c2=c_2(\F), \ k2=K_B^2,$$ 
$$kc1=K_Bc_1(\F), \ km1=K_BM_1, \ km2=K_BM_2, \ c1m1=c_1(\F)M_1, \ c1m2=c_1(\F)M_2.$$
Observe that, from 
$$\xi^{n-1}=\xi^{n-2}\pi^*c_1(\F)-\xi^{n-3}\pi^*c_2(\F)$$ 
we see that the following hold, for any line bundles $N, N'$ on $B$:
$$d=\xi^n=c_1(\F)^2-c_2(\F), \ \xi^{n-1}\pi^*N=c_1(\F)N, \ \xi^{n-2}\pi^*N \pi^*N'=NN'.$$
Now, since 
$$K_X = -(n-1)\xi + \pi^*K_B + \pi^*c_1(\F)$$ 
using \eqref{dete22}, we get that $c_1(\E)\xi^{n-1}=(\xi + \pi^*M_1)\xi^{n-1}=d+c_1(\F)M_1$ and $\frac{r}{2}(K_X+(n+1)\xi)\xi^{n-1}=\frac{r}{2}(3d+K_Bc_1(\F)+c_1(\F)^2)$. Thus, Lemma \ref{ulr}(v) gives (see \cite[Out(1)]{m1}) that
$$c_1(\F)M_1=\frac{1}{2}(-2d+rc_1(\F)^2+3rd+rK_Bc_1(\F)).$$
We now calculate $c_2(\E)\xi^{n-2}$. First observe that
\begin{equation}
\label{c212}
c_2(\E)=[Z]=[Y][\overline Z]=(\xi+\pi^*M_1)\pi^*M_2=\xi\pi^*M_2+\pi^*M_1\pi^*M_2
\end{equation}
so that
$$c_2(\E)\xi^{n-2}=c_1(\F)M_2+M_1M_2.$$
Next, using the standard exact sequences for $\Omega_{X/B}$ and $\Omega_X$, we find that 
\begin{equation}
\label{c2x}
c_2(X)=\pi^*c_2(B)+\pi^*c_2(\F)-(n-2)\xi\pi^*c_1(\F)+\binom{n-1}{2}\xi^2+\pi^*(K_Bc_1(\F))-(n-2)\xi\pi^*K_B.
\end{equation}
Moreover, we have
$$c_1(\E)^2 \xi^{n-2}=d+M_1^2+2c_1(\F)M_1$$
$$c_1(\E)K_X \xi^{n-2}=-d(n-2)+K_Bc_1(\F)+c_1(\F)^2+(3-n)c_1(\F)M_1+K_BM_1$$
$$K_X^2 \xi^{n-2}=(n-2)^2d+K_B^2+(6-2n)K_Bc_1(\F)+(5-2n)c_1(\F)^2$$
$$c_2(X)\xi^{n-2}=c_2(B)+c_2(\F)-(n-2)c_1(\F)^2+\binom{n-1}{2}d+(3-n)K_Bc_1(\F).$$
Using the above, \eqref{c212} and Lemma \ref{ulr}(vi), we find an identity \cite[Out(3)]{m1} from which, replacing  replacing $c_1(\F)M_1$, we get the identity $\rm Out(7)=0$ \cite[Out(7)]{m1}. 

We now compute $c_3(\E)\xi^{n-3}$. We have, using also \eqref{c212} and \eqref{c2x}:
$$c_1(\E)c_2(\E)\xi^{n-3}=c_1(\F)M_2+2M_1M_2, \ c_1(\E)^3\xi^{n-3}=d+3c_1(\F)M_1+3M_1^2$$
$$c_1(\E)^2K_X\xi^{n-3}=-(n-2)d+K_Bc_1(\F)+c_1(\F)^2-2(n-3)c_1(\F)M_1+2K_BM_1-(n-2)M_1^2$$ 
$$c_2(\E)K_X\xi^{n-3}=-(n-3)c_1(\F)M_2+K_BM_2-(n-2)M_1M_2$$ 
$$\begin{aligned}[t] 
c_1(\E)K_X^2\xi^{n-3} = & (n-2)^2d+K_B^2+(5-2n)c_1(\F)^2+2(3-n)K_Bc_1(\F)
+(n-2)(n-4)c_1(\F)M_1 \\ & -2(n-2)K_BM_1
\end{aligned}$$
$$\begin{aligned}[t]  
c_1(\E)c_2(X)\xi^{n-3} = & c_2(B)+c_2(\F)-(n-2)c_1(\F)^2+\binom{n-1}{2}d+(3-n)K_Bc_1(\F) \\
& +[2-n+\binom{n-1}{2}]c_1(\F)M_1-(n-2)K_BM_1
\end{aligned}$$
$$\begin{aligned}[t] 
K_Xc_2(X)\xi^{n-3} = & -(n-2)c_2(B)-(n-2)c_2(\F)+[(n-2)(n-3)+\binom{n-1}{2}]c_1(\F)^2 \\
& -(n-2)\binom{n-1}{2}d +[(n-2)(n-5)+\binom{n-1}{2}]K_Bc_1(\F)-(n-2)K_B^2.
\end{aligned}$$
Using the above, $c_3(\E)=0$ and Lemma \ref{ulr}(vii), we find an identity $\rm Out(11)=0$ \cite[Out(11)]{m1} from which, replacing $c_1(\F)M_1$, we get the identity $\rm Out(15)=0$ \cite[Out(15)]{m1}. From $\rm Out(7)=0$, we compute $M_1M_2$ \cite[Out(17)]{m1} and replacing it into $\rm Out(15)=0$, we get the new identity $\rm Out(18)=0$ \cite[Out(18)]{m1}. Now, \eqref{laj22} is
$$\chi(\O_B)+\frac{1}{2}M^2-\frac{1}{2}K_BM=0$$
that is
$$\frac{1}{12}(K_B^2+c_2(B))+\frac{1}{2}M_1^2-M_1M_2-\frac{1}{2}K_BM_1+\frac{1}{2}K_BM_2=0.$$
This identity is \cite[Out(19)]{m1} and, replacing $M_1M_2$ as above, we get the identity $\rm Out(23)=0$ \cite[Out(23)]{m1}, from which we compute $c_2(B)$ \cite[Out(24)]{m1}. Replacing $c_2(B)$ into the identity $\rm Out(18)=0$ \cite[Out(18)]{m1} and substituting $c_2(\F)=c_1(\F)^2-d$, we get a new identity $\rm Out(29)=0$ \cite[Out(29)]{m1}, namely
\begin{equation}
\label{ult}
-2c_1(\F)^2+2d-15dr+3dnr=0.
\end{equation} 
Next, we observe that 
\begin{equation}
\label{nuo}
H^i(\O_X(-Y+\overline Z))=0 \ \hbox{for all} \ i \ge 0.
\end{equation}
In fact $-Y+\overline Z=-\xi-\pi^*M$ and $R^q \pi_*(-\xi-\pi^*M)=0$ for every $q \ge 0$. Therefore, setting $\L'=\xi+\pi^*M$, Lemma \ref{mapc}(ii) gives that 
\begin{equation}
\label{e-l}
h^0(\E(-\L'))=1
\end{equation}
and we find an exact sequence
\begin{equation}
\label{elle}
0 \to \L' \to \E \to \mathcal Q \to 0
\end{equation}
where $\mathcal Q$ is the quotient sheaf. We will now prove that, for any fiber $F \cong \P^{n-2}$ of $\pi$, we have
\begin{equation}
\label{res}
\E_{|F} \cong \O_F^{\oplus (r-1)} \oplus \O_F(1).
\end{equation}
In fact, note that $\E_{|F}$ is globally generated and $c_2(\E_{|F})=0$ by \eqref{c212}. Therefore, a general morphism $\varphi:  \O_F^{\oplus (r-1)} \to \E_{|F}$ gives rise to an exact sequence
\begin{equation}
\label{g}
0 \to \O_F^{\oplus (r-1)} \to \E_{|F} \to \G \to 0
\end{equation}
where $\G$ is the quotient sheaf. Since $c_2(\E_{|F})=0$, then $D_{r-2}(\varphi) = \emptyset$ by \cite[Statement (folklore)(i), \S 4.1]{ba} and \cite[Lemma 4.1]{blv}, hence $\G$ is a line bundle on $F$. Also, $c_1(\E_{|F})=\xi_{|F}$ by \eqref{dete22}, gives that $\G \cong \O_F(1)$ and then the sequence \eqref{g} splits because $H^1(\O_F^{\oplus (r-1)}(-1))=0$. This proves \eqref{res}. 

Next, we claim that the quotient $\mathcal Q$ in \eqref{elle} is a vector bundle. To see the latter, since the map $\L' \to \E$ in \eqref{elle} is given by a nonzero section $\sigma \in H^0(\E(-\L'))$, we need to show that $\sigma$ never vanishes. Suppose that $\sigma$ vanishes at a point $x \in X$ and let $F$ be a fiber passing through $x$. Then also $\sigma_{|F}$ vanishes at $x$. On the other hand, $\sigma_{|F} \in H^0(\E(-\L')_{|F})=H^0(\O_F(-1)^{\oplus (r-1)} \oplus \O_F)$ by \eqref{res}, that is $\sigma_{|F}=(0, \ldots, 0, \lambda)$ for some $\lambda \in \C$. Since $\sigma_{|F}$ vanishes at $x$ we get that $\lambda=0$, hence $\sigma_{|F}=0$. Now, consider the diagram, obtained from the Eagon-Northcott resolution of $\I_{Z/X}$, where $\psi$ is an isomorphism by \eqref{nuo} and $h^0(\E(-\L'))=1$ by \eqref{e-l}:
$$\xymatrix{\C \sigma \cong H^0(\E(-\L')) \ar[d] \ar[r]^{\hskip -.6cm \psi}_{\hskip -.6cm \cong} & H^0(\I_{Z/X}(\overline Z)) \cong \C \psi(\sigma) \ar[d]^{r_F} \\ H^0(\E(-\L')_{|F})\ar[r] & H^0(\I_{Z/X}(\overline Z) \otimes \O_F).}$$
Since $\sigma_{|F}=0$, we get that $r_F(\psi(\sigma))=0$. On the other hand, the exact sequence
$$0 \to \O_X(-Y+\overline Z) \to \I_{Z/X}(\overline Z) \to \O_Y \to 0$$
gives another diagram, where $r_Y$ is an isomorphism by \eqref{nuo},
$$\xymatrix{\C \psi(\sigma) \cong H^0(\I_{Z/X}(\overline Z)) \ar[d]^{r_F} \ar[r]_{\hskip -.2cm \cong}^{r_Y} & H^0(\O_Y) \cong \C r_Y(\psi(\sigma)) \ar[d]^{r'_F} \\ H^0(\I_{Z/X}(\overline Z) \otimes \O_F) \ar[r]^{\hskip .7cm f} & H^0(\O_{Y \cap F})}$$
where $r'_F(r_Y(\psi(\sigma))=f(r_F(\psi(\sigma))=0$ and therefore the map $r'_F$ is zero. But this is a contradiction since $Y \cap F$, being a hyperplane in $F$, is nonempty. We have thus proved that $\mathcal Q$ is a vector bundle. 

Since $\det \mathcal Q = \det \E-\L' = \pi^*M_2$ and $\mathcal Q$ is globally generated since $\E$ is, it follows from \cite[Lemma 5.1]{lo} that $\mathcal Q \cong \pi^* \H$, where $\H$ is a rank $r-1$ vector bundle on $B$. Also, $\H$ is globally generated by \cite[Exc.~5.1.29(b)]{liu}, $c_1(\H)=M_2$ and $c_1(\H)^2=M_2^2=0$, hence   $c_2(\H)=0$ by \cite[Lemma 4.2(i)]{blv}. Now, we obtain from \eqref{elle} the following exact sequence
$$0 \to \L'((1-n)\xi) \to \E((1-n)\xi) \to (\pi^*\H)((1-n)\xi) \to 0$$
and we note that, for each $i \ge n-2$, we have $H^{i+1}(\L'((1-n)\xi))=0$ because $\L'((1-n)\xi)=(2-n)\xi+\pi^*M$ and $R^q \pi_*((2-n)\xi)=0$ for every $q \ge 0$ and $H^i(\E((1-n)\xi))=0$ because $\E$ is Ulrich. It follows that $H^i((\pi^*\H)((1-n)\xi))=0$ for $i \ge n-2$. Since $R^q \pi_*((1-n)\xi)=0$ for every $q \ge 0, q \ne n-2$ and $R^{n-2} \pi_*((1-n)\xi) \cong \O_B(-\det \F)$, we deduce that $H^{i-n+2}(\H(-\det \F))=0$ for $i \ge n-2$. Hence $\chi(\H(-\det \F))=0$ that is, by Riemann-Roch,
\begin{equation}
\label{chim3}
\frac{r-1}{12}(K_B^2+c_2(B))+\frac{r-1}{2}c_1(\F)^2-c_1(\F)M_2-\frac{1}{2}K_BM_2+\frac{r-1}{2}K_Bc_1(\F)=0.
\end{equation}

Finally, we use \eqref{chim3}, that is the identity $\rm Out(33)=0$ \cite[Out(33)]{m1}. From the latter, replacing $c_2(B)$ and again $c_2(\F)=c_1(\F)^2-d$, we get the identity $\rm Out(37)=0$ \cite[Out(37)]{m1}, that is
\begin{equation}
\label{ult2}
(c_1(\F)^2+6d-dn)r=0.
\end{equation} 
In particular, the above implies that $n \ge 7$, for if $n \le 6$, then \eqref{ult2} is $(c_1(\F)^2+(6-n)d)r=0$, a contradiction since $\F$ is very ample and therefore $c_1(\F)^2>0$. 
Combining it with \eqref{ult} we find that $d(14-2n-15r+3nr)=0$, that is $n(3r-2)=15r-14$, a contradiction since $n \ge 7$. 
\end{proof}

\section{Examples}

We give here some examples of the cases appearing in Theorem \ref{c1^4=0} and Corollary \ref{c22=0}.

There are many examples Ulrich bundles $\E$ with $c_1(\E)^4=0$, both linear Ulrich triples and not, see \cite{lo}, \cite[Thm.~1]{lms1}. We give here one that is also a Del Pezzo fibration.

\begin{example} 
\label{secondo}

Let $B$ be a smooth irreducible curve and let $L$ be a very ample line bundle on $B$. Let $X = B \times \P^2 \times \P^2$ with projections $\pi_1 : X \to B, \pi_2 : X \to \P^2 \times \P^2$ and let $\O_X(1) = L \boxtimes (\O_{\P^2}(1) \boxtimes \O_{\P^2}(1))$. We have that $K_X+3H  = \pi_1^*(K_B+3L)$, hence $\pi_1=\varphi_{K_X+3H}$. Since $K_B+3L$ is very ample, we see that $\varphi_{K_X+3H}$ gives a Del Pezzo fibration on $X$ over $B$ with all fibers $\P^2 \times \P^2$. Let $\H$ be any Ulrich bundle for $(B, L)$ and let $\E= \H(4L) \boxtimes q^*(\O_{\P^2}(2))$, where $q: \P^2 \times \P^2 \to \P^2$ is a projection. Then $\E$ is an Ulrich bundle on $X$ by \cite[(3.5)]{be} and \cite[Cor.~4.9]{ls}. Moreover, $c_1(\E)^4=0, c_1(\E)^3 \ne 0$. However note that $(X,\O_X(1),\E)$ is also a linear Ulrich triple over $B \times \P^2$, as $\E \cong (\pi_1 \times q)^*(\G(\det \F))$ where $\F=(L \boxtimes \O_{\P^2}(1))^{\oplus 3}$ and $\G=\H(L)\boxtimes \O_{\P^2}(-1)$.
\end{example}

We also have the following.

\begin{example}
\label{esbs}

Let $T$ be a smooth irreducible $3$-fold, let $L$ be a very ample line bundle on $T$ such that $T \subset \P H^0(L)$ does not contain lines and let $\G$ be a rank $r \ge 2$ Ulrich bundle for $(T, L)$. Let $X=(\P^2 \times T) \cap H$ be a hyperplane section of the Segre embedding of $\P^2 \times T \subset \P^N$, that is embedded with $\O_{\P^2}(1) \boxtimes L$. Let $\O_X(1)=(\O_{\P^2}(1) \boxtimes L)_{|X}$ and $p:= {\pi_2}_{|X} : X \to T$ be the restriction of the second projection. 

We claim that:
\begin{itemize}
\item[(i)] $p : X \to T$ exhibits $(X, \O_X(1))$ as a scroll over $T$ (that is such that $K_X+2H_1=p^*\L$, with $H_1 \in |\O_X(1)|$ and $\L$ ample on $T$ - in this case $\L=K_T+3L$) with general fiber a linear $\P^1$, but having $L^3$ fibers equal to a linear $\P^2$.
\item[(ii)] $X$ is not a linear $\P^{4-b}$-bundle over a variety of dimension $b$ with $1 \le b \le 3$.
\item[(iii)] $\E := p^*(\G(2L))$ is an Ulrich bundle on $X$ such that $c_1(\E)^3 \ne 0, c_1(\E)^4 = 0, c_2(\E) \ne 0, c_2(\E)^2=0$. 
\end{itemize}
\begin{proof}
Note that that $K_T+3L$ is ample by \cite[Prop.~7.2.2 and 7.2.3]{bs}, since $T$ does not contain lines. Then (i) follows from \cite[Ex.~14.1.5]{bs}. To see (ii), let $x=(z,y) \in X$ and set $f_x=p^{-1}(p(x))=p^{-1}(y)$ for the fiber of $p$ passing through $x$. Note that $f_x$ is either a line of type $(\P^2 \times \{y\})\cap H$ or a plane of type $\P^2 \times \{y\}$. Assume that $X$ is a linear $\P^{4-b}$-bundle $\pi : X \to B$ over a smooth variety $B$ of dimension $b$ with $1 \le b \le 3$. Let $x \in X$ be general and let $F_x=\P^{4-b}$ be the fiber of $\pi$ passing through $x$. Now, any line $R$ contained in $F_x$ and passing through $x$ is such that $R \subset \P^2 \times T \subset \P^N$, hence, as is well-known and using the fact that $T$ does not contain lines, $R=R' \times \{y\}$ with $y=p(x) \in T$ and $R' \subset \P^2$ is a line. Hence $R=R' \times \{y\} \subset (\P^2 \times \{y\})\cap H$. Since this is true for every $R$ such that $x \in R \subseteq F_x$, we deduce that $F_x \subseteq f_x$. But $f_x$ is a line, hence so is $F_x$ and therefore $b=3$. Let $M=\P^2 \times \{y_0\}$ be a linear $\P^2$ contained in $X$ and let $x_0 \in M$ be a general point. Since $F_{x_0} \subsetneq M$ and $\pi$ contracts $F_{x_0}$ to a point, while $\pi$ does not contract $M$ to a point, we have that $\pi(M)$ is a curve and we get a surjective morphism $\pi_{|M} : M \to \pi(M)$. But there is no surjective morphism from $\P^2$ to a curve. This contradiction proves (ii).  To see (iii), setting $\E' = \pi_2^*(\G(2L))$, we get that $\E'$ is an Ulrich  bundle on $\P^2 \times T$ by \cite[Lemma 4.1(ii)]{lo}. Hence $\E = p^*(\G(2L))= \E'_{|X}$ is an Ulrich  bundle on $X$ such that 
$$c_1(\E)^3 = p^*((c_1(\G)+2rL)^3)=p^*(c_1(\G)^3+6rc_1(\G)^2L+12r^2c_1(\G)L^2+8r^3L^3)\ne 0$$ 
since $L$ is very ample, and $c_1(\E)^4 = 0$. Clearly $c_2(\E)^2=p^*(c_2(\G(2L))^2)=0$. On the other hand, $c_2(\E) \ne 0$ by (ii) and \cite[Thm.~2]{blv}. This proves (iii).
\end{proof}

Note that, in the above example, we might have either that $c_3(\E)=0$ or that $c_3(\E) \ne 0$.

To see this, let $T \subset \P^9$ be the Veronese $3$-fold, that is $(T,L)=(\P^3,\O_{\P^3}(2))$. Let $\N$ be the null-correlation bundle on $\P^3$, that is (see for example \cite[\S 4.2]{oss}) the rank $2$ bundle sitting in the exact sequence
$$0 \to \N \to T_{\P^3}(-1) \to \O_{\P^3}(1) \to 0.$$
It is known (and also easy to check) that $\G:=\N(2)$ is a rank $2$ Ulrich bundle for $(T,L)$ (see for example \cite[Ex.~6.5]{be}). Now we have 
$$\E = p^*(\G(2L))=p^*(\N(6))$$ 
hence
$$c_1(\E)=12p^*H_{\P^3}, \ c_2(\E)=37p^*H_{\P^3}^2, c_3(\E)=0.$$
On the other hand, also $\E^{\oplus 2}$ is Ulrich with 
$$c_2(\E^{\oplus 2})=c_1(\E)^2+2c_2(\E)=218p^*H_{\P^3}^2 \ne 0, \ c_2(\E^{\oplus 2})^2=0$$
and 
$$c_3(\E^{\oplus 2})=2c_3(\E)+2c_1(\E)c_2(\E)=888p^*H_{\P^3}^3=888F \ne 0$$ 
where $F$ is a general fiber of $p$.
\end{example}

\section{Ulrich bundles with $c_3(\E)=0$}

In this section we study Ulrich bundles with $c_3(\E)=0$. 

First, we take care of the case $c_1(\E)^3=0$, that is a simple consequence of the results in \cite{ls}.

 
\begin{lemma}
\label{c13}
Let $X \subset \P^N$ be a smooth irreducible variety of dimension $n \ge 3$ and let $\E$ be a rank $r$ vector bundle on $X$. Then $\E$ is Ulrich with $c_1(\E)^3=0$ if and only if $(X,\O_X(1),\E)$ is one of the following:
\begin{itemize}
\item [(i)] $(\P^n,\O_{\P^n}(1),\O_{\P^n}^{\oplus r})$.
\item [(ii)] A linear Ulrich triple over a curve or a surface.
\end{itemize} 
In particular we have that $c_3(\E)=c_2(\E)^2=0$.
\end{lemma}
\begin{proof}
If $(X,\O_X(1),\E)$ is as in (i) or (ii), then $\E$ is Ulrich in the first case by \cite[Prop.~2.1]{es} (or \cite[Thm.~2.3]{be}), in the second case by \cite[Lemma 4.1]{lo}. Moreover $c_1(\E)^3=c_3(\E)=c_2(\E)^2=0$ since, by definition of linear Ulrich triple, $\E$ is pull-back of a bundle on a variety of dimension at most $2$.

Now assume that $\E$ is Ulrich and $c_1(\E)^3=0$. If $(X,\O_X(1))=(\P^n,\O_{\P^n}(1))$, we are in case (i) by Lemma \ref{ulr}(iii). Thus $(X,\O_X(1)) \ne (\P^n,\O_{\P^n}(1))$ and $\rho(X) \ge 2$, for otherwise $c_1(\E)^3=0$ implies that $c_1(\E)=0$ and then $(X,\O_X(1))=(\P^n,\O_{\P^n}(1))$ by \cite[Lemma 2.1]{lo}. 

If $n \ge 4$, we can repeat the proof of \cite[Cor.~4]{ls} using now the fact that $c_1(\E)^3=0$ and that $3 \le \lfloor \frac{n}{2} \rfloor +1$. It follows from that proof that $(X,\O_X(1),\E)$ is as in (ii). If $n=3$, \cite[Rmk.~2.2 and Thm.~2]{lm} give that $(X, \O_X(1), \E)$ is as in (ii). 
\end{proof}

We now prove Theorem \ref{trecasi}.

\renewcommand{\proofname}{Proof of Theorem \ref{trecasi}}
\begin{proof}
Observe that $\Pic(X) \not\cong \Z \O_X(1)$, for otherwise \cite[Lemma 7.6(ii) and Cor.~1]{blv} imply that $c_3(\E) \ne 0$. In particular $d=\deg X \ge 3$ and, if $(X, \O_X(1))$ is a Del Pezzo $n$-fold, then, using the classification of Del Pezzo manifolds (see for example \cite[pages 860-861]{lp}, \cite[Table, page 710]{f1}) we see that only the cases $d \in \{6, 7\}$ are possible. But, in these cases, $(X, \O_X(1))$ is anyway as in (b). 

In order to prove the first assertion of the theorem, namely that $(X, \O_X(1))$ is one of (a)-(c), it is therefore enough, by Proposition \ref{abc}, to show that $X_3$ is covered by lines. Indeed, set $\E' = \E_{|X_3}$, so that also $\E'$ is Ulrich and then globally generated by Lemma \ref{ulr}(iv)-(i). Since $c_3(\E')=c_3(\E)_{|X_3}=0$, it follows from \cite[Prop.~4.3(i)]{blv} that $\B_+(\E')=X_3$. On the other hand, we know by \cite[Thm.~2]{bu} that 
\begin{equation}
\label{cop}
\B_+(\E')=\bigcup_L L
\end{equation}
where $L$ ranges over all lines $L \subset X_3$ such that ${\E'}_{|L}=\E_{|L}$ is not ample and this proves that $X_3$ is covered by lines, hence the first part of the theorem.

Suppose now that $c_2(\E)^2=c_3(\E)=0$ and $n \ge 4$, so that we have a morphism $\pi : X \to B$ as in (a), (b) or (c). 

First, consider the case $c_2(\E)=0$. It follows from \cite[Thm.~2]{blv} that $(X, \O_X(1), \E)$ is a linear Ulrich triple $\pi' : X \to B'$ over a smooth curve $B'$. Now, the fibers of $\pi'$ are $\P^{n-1}$'s, hence they are contracted by $\pi$, that is they are contained in a fiber of $\pi$. This implies that we are in case (a) and, similarly, the fibers of $\pi'$ are contracted by $\pi'$, hence $\pi=\pi'$ and $(X,\O_X(1),\E)$ is a linear Ulrich triple over $B$. 

Now assume that $c_2(\E) \ne 0$. If we are in case (a), we conclude by Lemma \ref{pf}, if we are in case (b), we conclude by Lemma \ref{sup}, while if we are in case (c), then $n=4$ and $B \cong \P^1$ by Lemma \ref{qf}.
\end{proof}
\renewcommand{\proofname}{Proof}

\begin{remark}
In rank $2$, we have that $c_3(\E)=0$ and if $c_2(\E)^2=0$, it might well be that $(X, \O_X(1))$ is not as in (a)-(c) of Theorem \ref{trecasi}. In fact, if we pick $X$ as in Example \ref{esbs} with $T \subset \P^9$ the Veronese $3$-fold, it is easy to show that there are no line bundles $\L$ on $X$ with $\L^2=0$, hence $X$ does not have a morphism onto a curve and $(X, \O_X(1))$ is not as in (b) by Example \ref{esbs}(ii). 
\end{remark}

\section{On connectedness of Ulrich subvarieties}

In this section we will prove Corollary \ref{conn}, Corollary \ref{conn2} and we will give some related remarks.

\renewcommand{\proofname}{Proof of Corollary \ref{conn}}
\begin{proof}
By hypothesis we have that $c_2(\E) \ne 0$, hence $d=\deg X \ge 2$ by Lemma \ref{ulr}(iii). Since Ulrich subvarieties associated to $\E$ are normal (see Definition \ref{usv}), they are irreducible if and only if they are connected. Now, if there is a disconnected Ulrich subvariety, then $c_3(\E)=0$ by \cite[Prop.~6.1 and Cor.~1]{blv} and we get a contradiction by Theorem \ref{trecasi}.
\end{proof}
\renewcommand{\proofname}{Proof}

In cases (a)-(c) of Theorem \ref{conn}, the situation is more variegated, as Ulrich subvarieties may be both all connected and all disconnected.

\begin{example}
\label{a}
In case (a), let $(X, \O_X(1))$ be $(\P(\F), \O_{\P(\F)}(1))$, where $\F$ is a very ample rank $n \ge 3$ vector bundle on a smooth curve $C$ of genus $g$. In \cite[Ex.~10.1]{blv} there is an example where all Ulrich subvarieties are disconnected. On the other hand, let $M$ be a general divisor of degree $g-1$ on $C$ and let $\E=(\O_{\P(\F)}(1) \otimes \pi^\ast M)^{\oplus r}$, where $\pi$ is the bundle map. Then $\E$ is Ulrich and very ample, so that all Ulrich subvarieties associated to $\E$ are connected by \cite[Thm.~II(b)]{fl}.
\end{example}

\begin{example}
\label{b}
In case (b), let $S \subset \P^N$ be any smooth surface not containing any line and let $(X, \O_X(1)) = (\P^{n-2} \times S, \O_{\P^{n-2}}(1) \boxtimes \O_S(1))$. Let $\G$ be a rank $r \ge 3$ Ulrich bundle on $S$ and let $\E = \O_{\P^{n-2}}(2) \boxtimes \G$. Then $\E$ is a very ample Ulrich bundle on $(X, \O_X(1))$ by \cite[Rmk.~4.3(i)]{ls}, hence all its associated Ulrich subvarieties are connected by \cite[Thm.~II(b)]{fl}. On the other hand, in case (b), one can take a linear Ulrich triple over the surface and then $c_3(\E)=0$ and all Ulrich subvarieties are disconnected by \cite[Prop.~6.1 and Cor.~1]{blv}.
\end{example}

\begin{example}
\label{c}
In case (c), one can take $Q=Q_{n-1}, (X, \O_X(1)) = (\P^1 \times Q, \O_{\P^1}(1) \boxtimes \O_Q(1))$ and $\E = \pi_2^*(\mathcal S(1))^{\oplus r}, r \ge 3$, where $\pi_2: X \to Q$ is the second projection and $\mathcal S$ is a spinor bundle on $Q$. Then $\E$ is Ulrich by \cite[(3.5)]{be} and, if $n \ge 4$, then $c_3(\E) \ne 0$, hence all Ulrich subvarieties are connected by \cite[Prop.~6.1 and Cor.~1]{blv}. If $n=3$, then $c_3(\E)=0$, hence all Ulrich subvarieties are disconnected by \cite[Prop.~6.1 and Cor.~1]{blv}.
\end{example}

When $c_2(\E)^2 = 0$, we can prove that, in many cases, Ulrich subvarieties are disconnected.

\begin{lemma} 
\label{disc}
Let $n \ge 2$ and let $\E$ be a rank $r$ Ulrich bundle on $X \subset \P^N$ such that $c_2(\E) \ne 0$ and either one of the following is satisfied:
\begin{itemize}
\item [(i)] $c_2(\E)^2 = 0$ and one of $\{c_1(\E)^4, c_1(\E)^2c_2(\E), c_4(\E)\}$ is not zero, or 
\item [(ii)] $c_1(\E)^3=0$.
\end{itemize}
Then $c_3(\E)=0$ and all Ulrich subvarieties associated to $\E$ have at least $r-1$ connected components. 
\end{lemma}
\begin{proof}
Since $c_2(\E) \ne 0$, then $r \ge 2$ and if $r=2$ the conclusion is obvious. Now suppose that $r \ge 3$. If (i) holds, then $c_3(\E)=0$ by Lemmas \ref{ulr}(i) and \ref{c3no}(ii), while if (ii) holds, then $c_3(\E)=0$ by \cite[Lemma 4.2(i)]{blv}. Hence we can apply \cite[Cor.~1]{blv} to conclude. 
\end{proof}

We can now prove Theorem \ref{c2^2=0}.

\renewcommand{\proofname}{Proof of Theorem \ref{c2^2=0}}
\begin{proof}
(i) follows from \cite[Thm.~2]{blv}, (ii) from Theorem \ref{c1^4=0}, (iii) from Corollary \ref{c22=0} and (iv) from Theorem \ref{trecasi}.
\end{proof}
\renewcommand{\proofname}{Proof}

Finally, we prove Corollary \ref{conn2}.

\renewcommand{\proofname}{Proof of Corollary \ref{conn2}}
\begin{proof}
It follows from Corollary \ref{conn} that if $Z$ is disconnected and $r \ge 3$, then $(X, \O_X(1))$ is as in (a)-(c). 

Now assume that $(X, \O_X(1))$ is as in (a)-(c) of Corollary \ref{conn} and, if $n=4$, in case (b), $\det \E$ is ample, in case (c), either $B \not\cong \P^1$ or $B \cong \P^1$ and $\det \E$ is ample. In particular, case (c) is now excluded by Lemma \ref{qf}, while, in case (b), only (i)-(ii) of Lemma \ref{sup} remain.

Observe that if $(X,\O_X(1),\E)$ is a linear Ulrich triple over $B$, then $c_3(\E)=0$ and we are in case (b), for in case (a) we would have the contradiction $c_2(\E)=0$. Moreover, \cite[Lemma 4.6]{blv} implies that $Z$ is a disjoint union of linear $\P^{n-2}$'s and, if $r \ge 3$, $Z$ is disconnected by \cite[Cor.~1]{blv}.

Next, consider the cases $(X, \O_X(1), \E) \cong (\P^s \times \P^3, \O_{\P^s}(1) \boxtimes  \O_{\P^3}(1), \pi_2^*(\O_{\P^3}(s))^{\oplus r})$, $s \in \{1 ,2\}$. Note that $(X, \O_X(1), \E)$ is a linear Ulrich triple over $\P^3$ and, since $c_3(\E) \ne 0$, $Z$ is irreducible by \cite[Prop.~6.1 and Cor.~1]{blv}. It follows by \cite[Lemma 4.6]{blv} that $Z$ is a $\P^s$-bundle over a smooth irreducible curve $Z_B \subset \P^3$ defined by the $(r-2) \times (r-2)$'s minors of an $r \times (r-2)$ matrix of general linear forms.

To conclude, we can assume that $(X,\O_X(1),\E)$ is not a linear Ulrich triple over $B$ and $(X, \O_X(1), \E) \not\cong (\P^s \times \P^3, \O_{\P^s}(1) \boxtimes \O_{\P^3}(1), \pi_2^*(\O_{\P^3}(s))^{\oplus r})$, $s \in \{1 ,2\}$. 
Therefore, only case (a) remains.

It follows from the proof of Lemma \ref{pf} that $\det \E$ is $(n-4)$-ample, hence Lemma \ref{mapc} gives that $c_3(\E)=0$, so that $Z$ is disconnected when $r \ge 3$ by \cite[Cor.~1]{blv}, and there is a morphism with connected fibers $p : X \to C$ onto a smooth curve $C$. Moreover, there is a smooth irreducible divisor $Y \in |\det \E|$ with $Z \subset Y$ and a divisor $\Delta$ on $C$ such that if $\overline Z=p^* \Delta$, then $Z= \overline Z \cap Y$ and $Z$ is smooth. Hence $\Delta$ is reduced and therefore $\overline Z$ is a disjoint union of fibers of $p$. Now, for any fiber $\P^{n-1}$ of $\pi$, we have that the morphism $p_{|\P^{n-1}} : \P^{n-1} \to C$ must be constant, hence $p$ contracts all fibers of $\pi$ to a point and \cite[Lemma 1.15(b)]{de} gives a morphism $\psi : B \to C$ such that $p=\psi \circ \pi$. On the other hand, $p$ has connected fibers and therefore $\psi$ must be an isomorphim and we have that $p=\pi$. Since $Y \sim \xi + \pi^*M$ by Lemma \ref{pf}, $\overline Z=\pi^* \Delta$ and $Z= \overline Z \cap Y$, we get that $Z$ is a disjoint union of linear $\P^{n-2}$'s. 
\end{proof}
\renewcommand{\proofname}{Proof}

\begin{remark}
Let $X \subset \P^N$ be a smooth irreducible variety of dimension $n \ge 3$ and let $\E$ be a rank $r \ge 3$ Ulrich bundle on $X$ with $c_2(\E) \ne 0$ and $c_2(\E)^2=0$ and let $Z$ be a disconnected general Ulrich subvariety associated to $\E$. Even when $n \in \{3, 4\}$, in many cases, we can give the same description of $Z$ as the one in Corollary \ref{conn2}. In fact, if $c_1(\E)^3=0$, then $Z$ is a disjoint union of linear $\P^{n-2}$'s by Lemma \ref{c13} and \cite[Lemma 4.6]{blv}. Next, assume that $c_1(\E)^3 \ne 0$. Since $c_3(\E)=0$ by \cite[Prop.~6.1 and Cor.~1]{blv}, the proof of Lemma \ref{mapc} shows that there is a morphism with connected fibers $Y \to C$ onto a smooth curve, coming from to the linear system $|\O_Y(Z)|$. If $C$ is not rational, then, we can repeat the proof of \cite[Thm.~5.2.3]{bs}, using only the fact that $Y$ is smooth irreducible and $(n-3)$-big, because $\kappa(\det \E) \ge 3$. We get that the map $Y \to C$ extends to $X$ and we can reach the same conclusion on Ulrich subvarieties as the one in Corollary \ref{conn2}. 
\end{remark}

There is one case when we know that the curve $C$ in the above remark is rational: if $r=2$ and $H^1(\E^*)=0$, then \cite[(4.2)]{blv} implies that $h^0(\O_Y(Z))=2$ and the morphism is $Y \to \P H^0(\O_Y(Z))=\P^1$. This is exactly what happens in Example \ref{esbs} when $r=2$, since in the exact sequence
$$0 \to (\E')^*(-X) \to (\E')^* \to \E^* \to 0$$
we have that $H^1((\E')^*)=H^2((\E')^*(-X))=0$ by K\"unneth's formula, hence $H^1(\E^*)=0$.

\end{document}